\documentclass[12pt]{amsart}

\usepackage{color, amssymb}
\usepackage{amsthm}
\usepackage{tikz}
\usepackage{tikz-cd}
\usepackage{verbatim}
\usepackage{hyperref}
\usepackage{enumerate}
\usepackage{fullpage}
\usepackage{caption}
\usepackage{bigints}
\usepackage[textheight=8.94in, textwidth=6.8in]{geometry}
\usepackage{float}

\newtheorem{theorem}{Theorem}[section]
\newtheorem{lemma}[theorem]{Lemma}
\newtheorem{corollary}[theorem]{Corollary}
\newtheorem{proposition}[theorem]{Proposition}

\theoremstyle{definition}
\newtheorem{definition}[theorem]{Definition}
\newtheorem*{example}{Example}

\theoremstyle{remark}
\newtheorem{remark}[theorem]{Remark}

\numberwithin{equation}{section}

\def \l {\lambda}

\def\C{\mathbb C}

\def\A{\mathbb A}

\def\Tr{{\rm Tr}}
\def\Gal{{\rm Gal}}
\def\F{{\mathbb F}}
  
\def\Z{{\mathbb Z}}

\def\Q{{\mathbb Q}}

\def \R {\mathcal R}

\def\ord{{\mathrm ord}}

\def\({\left(}
\def\){\right)}
\def \l {\lambda}

\def\C{\mathbb{C}}
\def\R{\mathbb{R}}
\def\Z{\mathbb{Z}}
\def\Q{\mathbb{Q}}
\def\F{\mathbb{F}}

\newcommand{\fp}
{\mathbb{F}_p}

\catcode`,\active

\catcode`\,12

\catcode`,\active

\catcode`\,12

\catcode`,\active

\catcode`\,12

\catcode`,\active

\catcode`\,12

\DeclareMathOperator{\Spec}{Spec}
\DeclareMathOperator{\Frob}{Frob}
\DeclareMathOperator{\tr}{tr}
\DeclareMathOperator{\rank}{rank}
\DeclareMathOperator{\geom}{geom}
\DeclareMathOperator{\arith}{arith}
\DeclareMathOperator{\GL}{GL}
\DeclareMathOperator{\SL}{SL}
\DeclareMathOperator{\SU}{SU}
\DeclareMathOperator{\OO}{O}

\DeclareMathOperator{\Sym}{Sym}
\DeclareMathOperator{\Res}{Res}

\DeclareMathOperator{\Swan}{Swan}
\DeclareMathOperator{\Leg}{Leg}
\DeclareMathOperator{\Cl}{Cl}

\DeclareMathOperator{\sep}{sep}

\newcommand{\Mod}[1]{\ (\mathrm{mod}\ #1)}

\begin{document}

\title{Hypergeometric Distributions and joint families of elliptic curves}


\address{Department of Mathematics, Texas State University, San Marcos, TX 78666}
\email{bgrove@txstate.edu}

\address{Department of Mathematics, Louisiana State University, Baton Rouge, LA 70803}
\email{hsaad@lsu.edu}

\author{Brian Grove and Hasan Saad}

\begin{abstract}
Recently, the first author as well as the second author with Ono, Pujahari, and Saikia determined the limiting distribution of values of certain finite field ${_2F_1}$ and ${_3F_2}$ hypergeometric functions. These hypergeometric values are related to Frobenius traces of elliptic curves and their limiting distribution is determined using connections to the theory of modular forms and harmonic Maass forms. Here we determine the limiting distribution of values of some ${_4F_3}$ hypergeometric functions which are sums of traces of Frobenius for a pair of elliptic curves. To obtain this result, we generalize Michel's work on Sato--Tate laws for families of elliptic curves to the setting of pairs of families, and we show that a generic pair admits an independent Sato--Tate distribution as the finite field grows. To this end, we use various results from the theory of \'etale cohomology, Deligne's work on the Weil conjectures, and the work of Katz on monodromy groups. In the cases previously studied using modular methods, we elucidate the connection between the modular forms that appear and the machinery of \'etale cohomology.
\end{abstract}

\maketitle

\section{Introduction and Statement of Results}

In the 1980's, Greene \cite{GreenePhD, GreenePaper} defined hypergeometric functions over finite fields as an analogue of the classical hypergeometric functions. Using the analogy between character sum expansions and power series expansions, he proved various special evaluations and transformation formulas that closely mirror their classical counterparts. These functions have played central roles in the study of combinatorial supercongruences \cite{CS1,CS2,CS3,CS4,CS5,CS6}, elliptic curves \cite{EC1,EC2,EC3,EC4,EC5,EC6}, $K3$ surfaces \cite{K31,K32,EC5}, Dwork hypersurfaces \cite{Dwork1,Dwork2,Dwork3}, hyperelliptic curves \cite{hyper1,hyper2}, Calabi--Yau threefolds \cite{CY1,CY2}, the theory of hypergeometric motives and Galois representations \cite{HM1,HM2,HMM1,HMM2,FusEtAl}, and the theory of modular forms \cite{MM1,MM2,MM3,MM4,MM5,MM6, HLLT}.

The limiting distributions of these finite field hypergeometric functions have been studied in multiple works \cite{HGFDistributionOnoSaikia, 3F2Explicit, ono2024distributionhessianvaluesgaussian, grove2024hypergeometricmomentshecketrace}. To make this precise, we first recall the definition of these finite field hypergeometric functions. Let $p$ be a prime, $\zeta_p$ a fixed primitive $p$th root of unity, and let $\omega$ be a generator of the character group of $\F_p^\times.$ Furthermore, let $\alpha=\{a_1,\ldots,a_n\}$ and $\beta=\{1,b_2,\ldots,b_n\}$ be two collections of rational numbers and let $M$ be the least common multiple of the denominators of $a_1,\ldots,a_n,b_1,\ldots,b_n.$\footnote{We say that the hypergeometric datum $\{\alpha,\beta\}$ is of length $n.$} In this notation, if $p\equiv 1\pmod M$ and $\lambda\in\F_p^\times,$ then we have \cite{McCarthyTransformations} the finite field hypergeometric function\footnote{The function $H_p$ is a variant of the ${_{n+1}F_n}$ hypergeometric functions originally defined by Greene.}
\begin{align*}
H_p(\alpha,\beta|\lambda)&:=H_p\left(\begin{array}{cccc}
    a_1 & a_2  & \ldots & a_n  \\
     1 &  b_2 & \ldots &  b_n
\end{array}\bigg|\ \lambda\right) \\
&:=\frac{1}{1-p}\sum\limits_{k=0}^{p-2}\prod\limits_{i=1}^n\frac{g(\omega^{k+(p-1)a_i})g(\omega^{-k-(p-1)b_i})}{g(\omega^{(p-1)a_i})g(\omega^{-(p-1)b_i})}\cdot \omega^k((-1)^n\lambda),
\end{align*}
where
$$
g(\chi) := \sum\limits_{x\in\F_p^\times}\chi(x)\zeta_p^{x}
$$
is the Gauss sum corresponding to $\chi.$ Furthermore, we define $H_p(\alpha,\beta|0):=1.$ In this notation, the aforementioned works determine the limiting distribution of $H_p(\alpha,\beta|\lambda)$ as $p\to\infty$ and $\lambda$ varies over $\F_p$ for certain hypergeometric data $\{\alpha,\beta\}$ of lengths two and three.

Here we determine the limiting behavior of $H_p(\alpha,\beta|\lambda)$ for certain hypergeometric data $\{\alpha,\beta\}$ of length four.

\begin{remark}
For the hypergeometric data we consider, the condition $p\equiv 1\pmod M$ can be weakened to $\gcd(p,M)=1.$ (see Theorem 1.3 of \cite{BCM}). Therefore, since we are concerned with the asymptotic behavior as $p\to\infty,$ in what follows, we have no restrictions on $p.$ 
\end{remark} 

The following theorem gives the limiting moments.
\begin{theorem}\label{theorem1}
Let $d \in \{2,3,4,6\}, \alpha_d=\{\frac{1}{2d},1-\frac{1}{2d},\frac{1}{2d}+\frac{1}{2},-\frac{1}{2d}+\frac{1}{2}\},$ and $\beta=\{1,\frac{1}{2},1,\frac{1}{2}\}.$ If $m\geq 0,$ then we have
$$
\lim\limits_{p\to\infty}{p^{-m/2-1}}\sum\limits_{\lambda\in\F_p} H_{p}(\alpha_d,\beta|\lambda^2)^{m}= \begin{cases}
    C(m_1)C(m_1+1) &\ \ \ \text{ if } m=2m_1\ \text{ is even} \\
    0 &\ \ \ \text{ otherwise,}
\end{cases}
$$
where $C(n):=\frac{(2n)!}{n!(n+1)!}$ is the $n$-th Catalan number.
\end{theorem}

\begin{remark}
    The Catalan numbers $C(n)$ are ubiquitous in mathematics (\cite{CatalanBook}). In relation to our results, we have that these Catalan numbers
    arise  (see (5.3) of \cite{SatoTateGroups}) as the even moments of the traces of the Lie group $\SU_2.$ More precisely, if $n\geq 0,$ then we have
    $$
    \int_{\SU_2}\Tr(X)^{2n}dX = C(n),
    $$
    where the integral is taken with respect to the Haar measure. Equivalently, $C(n)$ is (see Lemma~\ref{oneRepMult}) the multiplicity of the trivial representation in the direct sum decomposition of the $(2n)$-th tensor power representation of $\SL_2.$
\end{remark}

As a corollary, we obtain the limiting distribution of the values $H_p(\alpha,\beta|\lambda^2)$ as $\lambda$ varies over $\F_p$ and $p\to\infty$ in terms of the Meijer $G$-functions $G_{2,2}^{2,0}$ (see Definition~\ref{MeijerGDef}).

\begin{corollary}\label{corollary1}
 Let $d \in \{2,3,4,6\}, \alpha_d=\{\frac{1}{2d},1-\frac{1}{2d},\frac{1}{2d}+\frac{1}{2},-\frac{1}{2d}+\frac{1}{2}\},$ and $\beta=\{1,\frac{1}{2},1,\frac{1}{2}\}.$ If $-4\leq a<b\leq 4,$ then we have
        $$
        \lim\limits_{p\to\infty}\frac{\#\left\{\lambda\in\F_p : p^{-1/2}\cdot H_p(\alpha_d,\beta|\lambda^2)\in[a,b] \right\}}{p} = \int_a^b\frac{4}{\pi |t|} G_{2,2}^{2,0} \left[\begin{matrix} 2 & 3 \\  \frac{1}{2} & \frac{3}{2}\end{matrix} \; \bigg| \; \frac{t^2}{16} \right]dt.
        $$
\end{corollary}

\begin{example}
For the prime $p=524287,$ the histogram illustrates the case $d=2$ of Corollary~\ref{corollary1}, i.e. the near match with the limiting distribution.
\begin{table}[H]
\begin{center}
\includegraphics[height=50mm]{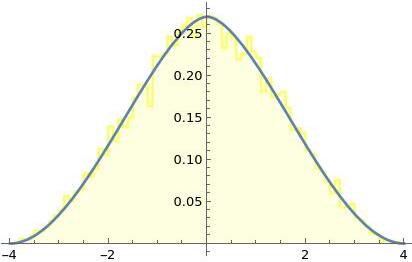}
\caption*{ \ \ \ $H_p(\alpha,\beta |\lambda^2)$ histogram for $d=2$ and $p=524287$}
\end{center}
\end{table}
\end{example}

Geometrically, the arithmetic of the $H_{p}$\footnote{In what follows, we denote by $H_{p}(\lambda)$ the hypergeometric functions in Theorem~\ref{theorem1} with $d$ understood from context.} functions in Theorem~\ref{theorem1} is intimately related to the arithmetic of certain families of elliptic curves $\widetilde{E}_{d,\lambda}$ (see Table~\ref{TableEC}). To make this precise, if $d\in\{2,3,4,6\},$ $p\geq 5$ is a prime, and $\lambda\in\F_p\setminus\{0,1\},$ then we have that (see Lemma~\ref{HpLegendreRelation})
\begin{equation}\label{eq:HGFLegCurveRelation}
H_p(\alpha_d,\beta|\lambda^2) = a_{p,d}(\lambda)+a_{p,d}(-\lambda),
\end{equation}
where $\alpha=\{\frac{1}{2d},1-\frac{1}{2d},\frac{1}{2d}+\frac{1}{2},-\frac{1}{2d}+\frac{1}{2}\},\beta=\{1,\frac{1}{2},1,\frac{1}{2}\},$ and where
$$
a_{p,d}(\lambda):= p + 1 -\#\widetilde{E}_{d,\lambda}(\F_p).
$$

In order to prove Theorem~\ref{theorem1}, we require a different set of tools than those used in \cite{HGFDistributionOnoSaikia, 3F2Explicit, ono2024distributionhessianvaluesgaussian, grove2024hypergeometricmomentshecketrace}. To make this precise, we shortly recall the method of proof in those works. To this end, recall that if $n\geq 0,$ then the $n$th Chebyshev polynomial of the second kind is defined as the unique polynomial $U_n(x)$ such that
\begin{equation}\label{ChebyshevPolynomialDefinition}
U_n(\cos\theta) = \frac{\sin((n+1)\theta)}{\sin(\theta)}.
\end{equation}
More explicitly, we have that
\begin{equation}\label{ChebyshevExplicit}
U_n(x):=\sum\limits_{k=0}^{\lfloor\frac{n}{2}\rfloor}(-1)^k\binom{n-k}{k}(2x)^{n-2k}.
\end{equation}
In this notation, if $d\in\{2,3,4\}$ and $m\geq 0,$ then the authors of \cite{HGFDistributionOnoSaikia,3F2Explicit,ono2024distributionhessianvaluesgaussian,grove2024hypergeometricmomentshecketrace} use either the theory of harmonic Maass forms or the theory of Hecke traces for triangle groups in \cite{HLLT} to express the twisted $m$th moment
$$
\frac{1}{p}\sum\limits_{\lambda\in\F_p}U_m\left(\frac{a_{p,d}(\lambda)}{2\sqrt{p}}\right)
$$
as a Fourier coefficient of a holomorphic cusp form. By Deligne's bound on the Fourier coefficients of cusp forms and induction, we have that if $m\geq 0,$ then the limiting $m$-th moment is given by
$$
\lim\limits_{p\to\infty}p^{-1-m/2}\sum\limits_{\lambda\in\F_p}a_{p,d}(\lambda)^m =\begin{cases}
    C(m_1) &\ \ \ \text{ if }m=2m_1\text{ is even } \\
    0 &\ \ \ \text{ otherwise}.
\end{cases}
$$
An application of the method of moments then gives the required distributions. In contrast to the methods described above, the moments of the $H_p$ functions we consider, or equivalently the mixed moments of $a_{p,d}(\lambda)$ and $a_{p,d}(-\lambda),$ cannot be described in terms of modular forms, which necessitates most of the work in this paper.

Here we use (\ref{eq:HGFLegCurveRelation}) in order to write the moments of $H_p$ as weighted combinations of mixed moments of the traces $a_{p,d}(\lambda)$ and $a_{p,d}(-\lambda).$ In order to apply the method of moments, we determine the asymptotics of the mixed moments for generic pairs of families of elliptic curves. To make this precise, for $i=1,2,$ let $a_{i,1}(\lambda),a_{i,2}(\lambda),a_{i,3}(\lambda),a_{i,4}(\lambda),a_{i,6}(\lambda)\in\Z[\lambda]$ and consider the two families of elliptic curves defined by
$$
E_{i,\lambda}:\ \ \ y^2+a_{i,1}(\lambda)xy+a_{i,3}(\lambda)y=x^3+a_{i,2}(\lambda)x^2+a_{i,4}(\lambda)x+a_{i,6}(\lambda).
$$
The family $E_{i,\lambda}$ can be put in a Weierstrass form (see III.1 of \cite{SilvermanAEC})
$$
E_{i,\lambda}:\ \ \ Y^2 = 4X^3-c_{i,4}(\lambda)X-c_{i,6}(\lambda),
$$
where $c_{i,4}(\lambda),c_{i,6}(\lambda)\in\Q(\lambda).$ In this notation, we let $\Delta_i(\lambda):= c_{i,4}(\lambda)^3-27c_{i,6}(\lambda)^2$ and $j_i(\lambda):=1728\cdot\frac{c_{i,4}(\lambda)^3}{\Delta_i(\lambda)}$ be the discriminant and $j$-invariant of $E_{i,\lambda}$ respectively. We say that $(E_{1,\lambda},E_{2,\lambda})$ is a generic pair if the following hold:
\begin{enumerate}
    \item $j_1(\lambda)$ and $j_2(\lambda)$ are nonconstant.
    \item There exists an $N\geq 1$ such that for $p\nmid N,$ the reductions of $E_{1,\lambda}$ and $E_{2,\lambda}$ mod $p$ have either good or multiplicative reduction at each $\lambda_0\in\overline{\F_p}.$
    \item There exists an $N\geq 1$ such that for $p\nmid N$ and $d\in\F_p^\times,$ the reductions of $E_{1,\lambda}$ and $E_{2,\lambda,d}$ mod $p$ are not isogenous over $\F_p(\lambda),$ where
    $$
    E_{2,\lambda,d}:\ \ \ dy^2 = 4x^3-c_{2,4}(\lambda)x-c_{2,6}(\lambda)
    $$
    is the quadratic twist of $E_{2,\lambda}$ by $d.$
\end{enumerate}
Finally, for $i=1,2$ and a prime $p,$ write $a_{i,p}(\lambda):=p+1-\#E_{i,\lambda}(\F_p).$  In this notation, we compute the limiting mixed moments of $a_{1,p}(\lambda)$ and $a_{2,p}(\lambda).$

\begin{theorem}\label{theorem2}
    If $(E_{1,\lambda},E_{2,\lambda})$ is a generic pair and $n,m\geq 0,$ then we have
    $$
    \lim\limits_{p\to\infty} p^{-1-(n+m)/2}\sum\limits_{\lambda\in\F_p} a_{1,p}(\lambda)^na_{2,p}(\lambda)^m = \begin{cases}
    C(n_1)C(m_1) & \ \text{\it if } n=2n_1 \text{ and } m=2m_1 \text{ are even} \\
    0 & \ \text{\it otherwise,}
    \end{cases}
    $$
    where $C(k)$ is the $k$th Catalan number.
\end{theorem}

Using these moments, we determine the limiting joint distribution of the pair $(a_{1,p}(\lambda),a_{2,p}(\lambda)).$

\begin{corollary}\label{corollary2}
    If $(E_{1,\lambda},E_{2,\lambda})$ is a generic pair, $-2\leq a_1<b_1\leq 2,$ and $-2\leq a_2<b_2\leq 2,$ then we have
    \begin{align*}
    &\lim\limits_{p\to\infty}\frac{\#\left\{\lambda\in\F_p: p^{-1/2}\cdot a_{1,p}(\lambda)\in[a_1,b_1]\text{ and }p^{-1/2}\cdot a_{2,p}(\lambda)\in[a_2,b_2]\right\}}{p} \\
    & = \left(\frac{1}{2\pi} \int_{a_1}^{b_1} \sqrt{4-x_1^2}\ dx_1\right) \cdot \left(\frac{1}{2\pi} \int_{a_2}^{b_2} \sqrt{4-x_2^2}\ dx_2\right)
    \end{align*}
\end{corollary}

\begin{remark}
The conditions of a generic pair of elliptic curves which is required for Theorem~\ref{theorem2} and Corollary~\ref{corollary2} can be relaxed (see Proposition~\ref{mixedMoments}) to the condition that
    $$
    R^1\pi_{1,!}\Q_\ell\not\cong_{\text{geom}} R^1\pi_{2,!}\Q_\ell\otimes\mathcal{L}
    $$
    for any rank $1$ $\ell$-adic system $\mathcal{L},$ where $\pi_{i}$ is the projection map $(x,y,\lambda)\to \lambda.$ For example (see Lemma~\ref{InertiaAction} and proof of Lemma~\ref{geomConstant}), if $E_{1,\lambda_0}$ has additive potentially multiplicative reduction for some $\lambda_0\in\A_{\F_p}$ while $E_{2,\lambda_0}$ has good reduction, this condition would be satisfied. In general, if one allows additive reduction, then the situation is more delicate and one needs to take into account twists by nonconstant rational functions.
\end{remark}

We remark that the methods used to prove Theorem~\ref{theorem2} and Corollary~\ref{corollary2} immediately offer a different - more geometric - proof of the limiting moments of the hypergeometric functions of length $2$ studied in \cite{HGFDistributionOnoSaikia,ono2024distributionhessianvaluesgaussian,grove2024hypergeometricmomentshecketrace}. Moreover, it also gives the case $d=6$ which is not accessible through the modular methods used in these works.

\begin{theorem}\label{theorem3}
    If $d\in\{2,3,4,6\}$ and $ m\geq 0,$ then we have
    $$
    \lim\limits_{p\to\infty}p^{-1-m/2}\sum\limits_{\lambda\in\F_p}H_p\left(\begin{array}{cc}
    \frac{1}{d} & \frac{d-1}{d}  \\
     1 & 1
\end{array}\bigg|\ \lambda\right)^m= \begin{cases}
    C(m_1) &\ \ \ \text{ if }m=2m_1\text{ is even} \\
    0 &\ \ \ \text{ otherwise.}
\end{cases}
    $$
\end{theorem}

Finally, with very slight modifications to our method, we offer an alternative proof of the limiting moments of the hypergeometric function of length $3$ studied in \cite{HGFDistributionOnoSaikia, 3F2Explicit}.

\begin{theorem}\label{theorem4}
    If $m\geq 0,$ then we have
    $$
    \lim\limits_{p\to\infty}p^{-1-m}\sum\limits_{\lambda\in\F_p}H_p\left(\begin{array}{ccc}
    \frac{1}{2} & \frac{1}{2} & \frac{1}{2} \\
     1 & 1 & 1
\end{array}\bigg|\ \lambda\right)^m= \begin{cases}
    \sum\limits_{i=0}^m(-1)^i\binom{m}{i}\frac{(2i)!}{i!(i+1)!} &\ \ \ \text{ if }m\text{ is even} \\
    0 &\ \ \ \text{ if }m\text{ is odd.}
\end{cases}
    $$
\end{theorem}

\begin{remark}
    We give the following remarks regarding Theorems~\ref{theorem2}, \ref{theorem3}, \ref{theorem4}, and Corollary~\ref{corollary2}.

    \begin{enumerate}
        \item We remark that Theorem~\ref{theorem2} and Corollary~\ref{corollary2} are well-known to the experts (for example, see Chapters 9 and 10 of \cite{KatzSarnak}). The work in this paper illustrates how the geometric results obtained by Deligne \cite{WeilII}, Katz \cite{KatzGKM,KatzESDE}, and others serve to obtain arithmetic-statistical results in number theory, as well as to explain the appearance of Catalan numbers in \cite{HGFDistributionOnoSaikia, 3F2Explicit, ono2024distributionhessianvaluesgaussian, grove2024hypergeometricmomentshecketrace}. Moreover, our methods, when restricted to the cases studied with modular methods, yield an explicit description of the ``mysterious'' cohomology groups that appear within this paper, elucidating the connection between the modular methods and the abstract machinery of \'etale cohomology in the situations where such a connection exists (through a universal elliptic curve over a modular curve; for example, see Sections 3 and 4 of \cite{HLLT}).

        \item We remark that Theorem~\ref{theorem3} follows immediately from the work of Michel \cite{RangMoyen}. We add it here for completeness  and to illustrate the connection between the geometric methods used here and the modular methods used in \cite{HGFDistributionOnoSaikia,ono2024distributionhessianvaluesgaussian,grove2024hypergeometricmomentshecketrace}.

        \item The methods used to prove Theorem~\ref{theorem4} explicitly bound the difference between the actual and limiting moments which leads to an explicit bound on the difference between the limiting and actual distributions of the $H_p$ functions \`a la \cite{3F2Explicit}. It is important to note that this yields no extra power-savings in the error compared to modular methods.

        \item Geometrically, the arithmetic of the $H_p$ functions in Theorem~\ref{theorem4} is related to twisted symmetric squares of the Clausen family of elliptic curves (see Lemma~\ref{hypergeometricClausen}).

        \item The moments in Theorem~\ref{theorem4} arise (see \cite{pastur2004moments}) as the even moments of the traces of the Lie group $\OO_3.$ More precisely, if $m\geq 0$ is even, then we have
        $$
        \int_{\OO_3}\Tr(X)^{m}dX = \sum\limits_{i=0}^{m}(-1)^i\binom{m}{i}\frac{(2i)!}{i!(i+1)!}.
        $$
    \end{enumerate}
\end{remark}

Our proof of Theorem~\ref{theorem2} closely mirrors the work of Michel on the equidistribution of Frobenius traces for a generic family of elliptic curves \cite[Proposition 1.1]{RangMoyen}. First, we describe $a_{1,p}(\lambda)$ and $a_{2,p}(\lambda)$ as traces of Frobenius on two $\ell$-adic sheaves $\mathcal{F}_1$ and $\mathcal{F}_2$ lisse on an open subset $U_p\subset\A_{\F_p}^1.$ In this notation, the $(m,n)$-th mixed moments are then described as sums of traces of Frobenius on the sheaves $\mathcal{F}_1^{\otimes n}\otimes\mathcal{F}_2^{\otimes n}.$ Under the conditions on a generic pair, we decompose these sheaves as a direct sum of irreducible representations of the so-called geometric monodromy group. The multiplicities of the trivial representation in these direct sum decompositions are either $0$ or products of Catalan numbers. This allows us to explicitly compute the contribution of the trivial parts to the moments.  The contribution of the non-trivial parts is then shown to be small using the Grothendieck--Lefschetz trace formula and Deligne's \cite{WeilII} results on the purity of cohomology groups in his seminal proof of the generalized Weil conjectures. The proof of Theorem~\ref{theorem3} follows from the intermediate propositions that lead up to proving Theorem~\ref{theorem2}, whereas Theorem~\ref{theorem4} follows from the same methods with slight modifications.

This paper is organized as follows. In Section~\ref{Section2}, we recall important facts on the relations between the $H_p$ functions mentioned above and families of elliptic curves. In Section~\ref{Section3}, we recall necessary background from the theory of \'etale cohomology and important facts on the first cohomology of elliptic curves. In Section~\ref{Section4}, we recall facts from representation theory. In Section~\ref{Section5}, we apply the results from \'etale cohomology and representation theory to compute the mixed moments of families of elliptic curves under a certain ``independence'' condition of their first cohomologies. In Section~\ref{Section6}, we describe this ``independence'' condition in terms of isogenies of families of elliptic curves viewed as elliptic curves over $\F_p(\lambda).$ In Section~\ref{Section7}, we recall the method of moments and derive the necessary distributions for Corollaries~\ref{corollary1} and \ref{corollary2}. In Section~\ref{Section8}, we conclude with the proofs of Theorems~\ref{theorem1} through \ref{theorem4} and Corollaries~\ref{corollary1} and \ref{corollary2}.

\section*{Acknowledgements}
The authors would like to thank Ling Long and Fang-Ting Tu for the numerous discussions and advice which highly contributed to the quality of this paper.
The authors would also like to thank Bill Hoffman and Ken Ono for their multiple comments, discussions, and helpful remarks. The authors also thank Wanlin Li who pointed out an error in an earlier draft. Finally, the authors thank the anonymous referee whose comments improved the quality of the exposition.
The first author was partially supported by a research assistantship from the Louisiana State University (LSU) Department of Mathematics.

\section{The \texorpdfstring{$H_p$}{Hp} functions and the arithmetic of elliptic curves}\label{Section2}

Here we recall important facts about the relations between $H_p$ functions and the arithmetic of one-parameter families of elliptic curves.

We first recall the relationship between the hypergeometric functions in Theorem~\ref{theorem3} and elliptic curves. If $d\in\{2,3,4,6\}$ and $p\geq 5$ is prime, define the families of elliptic curves $\widetilde{E}_{d,\lambda}$ by
\renewcommand{\arraystretch}{1.2}
\begin{table}[H]
\begin{center}
\scalebox{0.92}{
    \begin{tabular}{|c|c|}
    \hline
       $d$  & $\widetilde{E}_{d,\lambda}$  \\ \hline
       $2$  & $y^2=x(1-x)(x-\lambda)$ \\ \hline
       $3$ & $y^2+xy+\frac{\lambda}{27}y = x^3$ \\ \hline
       $4$ & $y^2=x(x^2+x+\frac{\lambda}{4})$ \\ \hline
       $6$ & $y^2+xy=x^3-\frac{\lambda}{432}$ \\ \hline
    \end{tabular}}
\end{center}
\caption{}\label{TableEC}
\end{table}
\renewcommand{\arraystretch}{1}
In this notation, the $H_p$ functions give the finite field point counts on these elliptic curves.

\begin{lemma}\cite{Koike, BCM, HLLT}\label{generalizedKoike}
    If $p\geq5$ is prime, $d\in\{2,3,4,6\},$ and $\lambda\in\F_p\setminus\{0,1\},$ then we have
    $$
    \#\widetilde{E}_{d,\lambda}(\F_p)=p+1-H_p\left(\begin{array}{cc}
    \frac{1}{d} & \frac{d-1}{d}  \\
     1 & 1
\end{array}\bigg|\ \lambda\right).
    $$
\end{lemma}

For the hypergeometric datum $\alpha=\{\frac{1}{2},\frac{1}{2},\frac{1}{2}\}, \beta=\{1,1,1\},$ the $H_p$ function is related to the finite field point counts on the Clausen elliptic curves
$$
E_\lambda^{\Cl}:\ \ \ y^2 = (x-1)(x^2+\lambda).
$$

\begin{lemma}[Th. 5 of \cite{EC5}]\label{hypergeometricClausen}
    If $p\geq 5$ is prime and $\lambda\in\F_p\setminus\{0,-1\},$ then we have
    $$
    H_p\left(\begin{array}{ccc}
    \frac{1}{2} & \frac{1}{2} & \frac{1}{2}  \\
     1 & 1 & 1
\end{array}\bigg|\ \frac{\lambda}{\lambda+1}\right) = \phi_p(\lambda+1)(a_p^{\Cl}(\lambda)^2-p),
    $$
    where $\phi_p$ is the Legendre character on $\F_p$ and where
    $$
    a_p^{\Cl}(\lambda):=p+1-\#E_\lambda^{\Cl}(\F_p).
    $$
\end{lemma}

Finally, we relate the $H_p$ functions in Theorem~\ref{theorem1} to the elliptic curves $\widetilde{E}_{d,\lambda}.$

\begin{lemma}\label{HpLegendreRelation}
    If $p\geq 3$ is prime, $d\in\{2,3,4,6\},$ $\alpha_d=\{\frac{1}{2d},1-\frac{1}{2d},\frac{1}{2d}+\frac{1}{2},-\frac{1}{2d}+\frac{1}{2}\},$ and $\beta=\{1,\frac{1}{2},1,\frac{1}{2}\},$ then we have $H_p(\alpha_d,\beta|\lambda)=0$ when $\lambda$ is not a square in $\F_p.$ Moreover, if $\lambda^2\neq 0,1$ then we have
    \begin{equation}\label{HGFEC}
    H_p\left(\alpha_d,\beta|\lambda^2\right)= a_{p,d}(\lambda)+a_{p,d}(-\lambda).
    \end{equation}
\end{lemma}

\begin{proof}
    If $\lambda$ is not a square, the fact that $H_p(\alpha,\beta|\lambda)=0$ follows immediately from Theorem 2.2 and Lemma 2.6 of \cite{splittingRoots}.
    To prove (\ref{HGFEC}) when $\lambda^2\neq 0,1,$ note that Corollary 2.3 and Lemma 2.6 of \cite{splittingRoots} imply that
    $$
    H_p\left(\begin{array}{cccc}
    \frac{1}{2d} & 1-\frac{1}{2d} & \frac{1}{2d}+\frac{1}{2} & -\frac{1}{2d}+\frac{1}{2}  \\
     1 & \frac{1}{2} & 1 & \frac{1}{2}
\end{array}|\lambda^2\right) = H_p\left(\begin{array}{cc}
    \frac{1}{d} & 1-\frac{1}{d} \\
    1 & 1
\end{array}\mid\lambda\right)+H_p\left(\begin{array}{cc}
    \frac{1}{d} & 1-\frac{1}{d} \\
    1 & 1
\end{array}\mid-\lambda\right).
    $$
    On the other hand, Lemma~\ref{generalizedKoike} gives that
    $$
    H_p\left(\begin{array}{cc}
    \frac{1}{d} & 1-\frac{1}{d} \\
    1 & 1
\end{array}\mid\lambda\right) = a_{p,d}(\lambda)
    $$
    and (\ref{HGFEC}) follows.
\end{proof}

\section{Background from \'Etale Cohomology}\label{Section3}

In order to prove Theorems~\ref{theorem1} through \ref{theorem4}, we describe the traces $a_{E,p}(\lambda)$ of an elliptic curve $E/\F_p(\lambda)$ as traces of Frobenius morphisms on an $\ell$-adic \'etale sheaf. This perspective allows us to use the resuls and heavy machinery of \'etale cohomology which we assume the reader is familiar with.  In Subsection~\ref{SubsectionGenericEtale}, for the reader's convenience, we recall the necessary background from the theory of \'etale cohomology. In Subsection~\ref{EtaleSheafFEC} we recall some background on the sheaves which encode finite-field point counts on families of elliptic curves. Those sheaves ``are'' essentially the rational $\ell$-adic Tate modules when the family is viewed as an elliptic curve defined over the function field (see Section~\ref{Section6}).

\begin{remark}
    For background on the applications of \'etale cohomology in analytic number theory, the reader is referred to \cite{AppliedladicCohomology}, \cite{traceFunctionsFFApplications}, \cite{StudySumsProducts}, \cite{KatzESFDEC}. For a more geometric perspective, the reader is referred to \cite{ECMilneBook}.
\end{remark}

\subsection{\'Etale sheaves and monodromy groups}\label{SubsectionGenericEtale}

Here we recall results from the theory of \'etale cohomology. More precisely, the results we recall allow us to view powers of Frobenius traces at different fibers simultaneously, and to then rewrite the sums of those powers in terms of the traces on the so-called cohomology group. For example, in \cite{HGFDistributionOnoSaikia,ono2024distributionhessianvaluesgaussian,grove2024hypergeometricmomentshecketrace}, certain weighted sums of moments of Frobenius traces are expressed explicitly as Fourier coefficients of certain cusp forms. The first cohomology groups that appear here should be thought of as (see Remark~\ref{Hc1ModRemark}) the geometric generalization of these forms when one does not have an underlying modular setting. We closely follow \cite{AppliedladicCohomology}. In what follows in this subsection, let $p$ be a prime, $k=\F_p, K=\F_p(X)$ the field of rational functions over $k,$ and let $\ell\neq p$ be a prime.

Recall that if $U\subset\A_{\F_p}^1,$ then an $\ell$-adic sheaf $\mathcal{F}$ that is lisse on $U$ ``is'' (more accurately, see Section 1 of Chapter V of \cite{ECMilneBook}) a continuous finite-dimensional Galois representation
$$
\rho_{\mathcal{F}}:G = \Gal(K^{\sep}/K)\to\GL(\mathcal{F}_{\overline{\eta}}),
$$
where $\mathcal{F}_{\overline{\eta}}$ is a finite-dimensional $\Q_\ell$-vector space of dimension $\rank(\mathcal{F}),$ and where $\rho_{\mathcal{F}}$ is unramified at points $u\in U.$ To make this precise, recall that for every $\lambda\in\A_{\F_p}^1,$ we have a decomposition group $D_\lambda,$ an inertia group $I_\lambda,$ a residue field $k(\lambda) = \F_{p^{\deg(\lambda)}},$ and an exact sequence
\begin{equation}\label{DIExactSequence}
\begin{tikzcd}
0 \arrow[r] & I_\lambda \arrow[r] & D_\lambda \arrow[r] & \Gal(\overline{k(\lambda)}/k(\lambda)) \arrow[r] & 0.
\end{tikzcd}
\end{equation}
In this notation, we say that $\rho_{\mathcal{F}}$ is unramified at $\lambda$ if $I_\lambda$ acts trivially on $\mathcal{F}_{\overline{\eta}}.$ Furthermore, for each $\lambda\in\A_{\F_p}^1,$ we denote by $\Frob_{k(\lambda)}^{\arith}$ the Frobenius morphism
$$
\begin{tikzcd}[row sep=-4pt, column sep= normal]
 \Frob_{k(\lambda)}^{\arith}\colon\!  \overline{k} \arrow[r] & \overline{k} \\
\ \ \ \ \ \ \ \ \ \quad u \arrow[r, maps to] & u^{\deg(\lambda)}       
\end{tikzcd}
$$
and by $\Frob_{k(\lambda)}^{\geom}$ the inverse of $\Frob_{k(\lambda)}^{\arith}.$ In this notation, for each $\lambda\in\A_{\F_p}^1,$ we have a lift $\Frob_\lambda$ of $\Frob_{k(\lambda)}^{\geom}$ to an $I_\lambda$-class of $D_\lambda.$ Furthermore, we denote by $\Frob_p$ the geometric Frobenius on $\overline{k},$ namely, the inverse of the arithmetic Frobenius $a\to a^p.$

In order to prove Theorems~\ref{theorem1} and \ref{theorem2}, we will describe the moments of the values $a_{E,p}(\lambda)$ as sums of the form $\sum\limits_{\lambda\in U(k)}\tr(\Frob_\lambda|\mathcal{F}_{\overline{\eta}})$ where $\mathcal{F}$ is an appropriately chosen \'etale sheaf. To determine the asymptotics of these moments, we express the sums over $U(k)$ as sums of traces of $\Frob_p$ acting on cohomology groups associated to the sheaf $\mathcal{F}.$ 

To make this precise, if $U\subset\A_{k}^1,$ denote by $U_{\overline{k}}$ the base-change of $U$ to $\overline{k}.$ If $\mathcal{F}$ is an $\ell$-adic sheaf that is lisse on $U,$ we have the associated cohomology groups with compact support $H_c^i(U_{\overline{k}},\mathcal{F}).$ These groups are finite-dimensional $\Q_\ell$-representations of $\Gal(\overline{k}/k).$ In this notation, the Grothendieck--Lefschetz trace formula expresses the sums of traces of Frobenius over a sheaf as an alternating sum of traces of Frobenius on the cohomology groups.

\begin{theorem}[Th. 4.1 of \cite{AppliedladicCohomology}, Th. 13.4 of \cite{ECMilneBook}]\label{GrothendieckLefschetz}
    If $\mathcal{F}$ is lisse on $U,$ then
    $$
    \sum\limits_{\lambda\in U(k)}\tr(\Frob_\lambda|\mathcal{F}_{\overline{\eta}}) = \sum\limits_{i=0}^2 \tr(\Frob_p|H_c^i(U_{\overline{k}},\mathcal{F})).
    $$
\end{theorem}

In order to estimate the sums $\sum\limits_{i=0}^2 \tr(\Frob_p|H_c^i(U_{\overline{k}},\mathcal{F})),$ we require bounds on both the eigenvalues of Frobenius and on the dimension of the cohomology groups. To this end, let $\mathcal{F}$ be an $\ell$-adic sheaf lisse on $U$ with $\rank(\mathcal{F})=n.$ We say that $\mathcal{F}$ is punctually pure of weight $w$ if for each $\lambda\in U,$ the eigenvalues $\alpha_{\lambda,1},\ldots,\alpha_{\lambda,n}$ of $\Frob_\lambda$ are such that $|\iota(\alpha_{\lambda,i})|=p^{\deg(\lambda)\cdot w/2}$ for every embedding $\iota:\overline{\Q_\ell}\to\C.$ In this notation, the fundamental theorem of Deligne (see III of \cite{WeilII}) bounds the eigenvalues of Frobenius on the middle cohomology.

\begin{theorem}[Th. 4.6 of \cite{AppliedladicCohomology}]\label{DelignePurity}
    If $\mathcal{F}$ is an $\ell$-adic sheaf lisse on $U\subset\A_{k}^1$ that is punctually pure of weight $0$  and $\alpha$ is an eigenvalue of $\Frob_p|H_c^1(U_{\overline{k}},\mathcal{F}),$ then $|\iota(\alpha)|\leq p^{1/2}$ for all embeddings $\iota:\overline{\Q_\ell}\to\C.$
\end{theorem}

It turns out that the extremal cohomology groups $H_c^0$ and $H_c^2$ are easy to express in terms of the representation of the geometric Galois group. To make this precise, let $G=\Gal(K^{\sep}/K).$ We then have an induced exact sequence
\begin{equation}\label{ArithGeoExactSequence}
\begin{tikzcd}
0 \arrow[r] & G^{\geom} \arrow[r] & G \arrow[r] & \Gal(\overline{k}/k) \arrow[r] & 0,
\end{tikzcd}
\end{equation}
where $G^{\geom} = \Gal(K^{\sep}/\overline{k}\cdot K)$ with $\overline{k}\cdot K$ the compositum of $\overline{k}$ and $K.$ We say that $\mathcal{F}$ is geometrically irreducible (resp. geometrically constant) if $\mathcal{F}_{\overline{\eta}}$ is irreducible (resp. trivial) as a representation of $G^{\geom}.$ Furthermore, if $\mathcal{F}$ is an $\ell$-adic sheaf lisse on $U$ and $w\in\Z,$ then up to taking a finite extension of $\Q_\ell,$ we denote by $\mathcal{F}(\frac{w}{2})$ the ``half''-Tate twist (see Remark 3.11 of \cite{AppliedladicCohomology}) of $\mathcal{F}$ by $\Q_\ell(\frac{w}{2}).$ In this notation, we have the following description of $H_c^0$ and $H_c^2.$

\begin{lemma}[Section 4.1 of \cite{AppliedladicCohomology}, Remark 2.4 of \cite{ECMilneBook}]\label{extremalCohomologyDescription}
    If $\mathcal{F}$ is an $\ell$-adic sheaf lisse on $U\subset\A_{k}^1,$ then we have that
    $$
    H_c^0(U_{\overline{k}},\mathcal{F}) = 0
    $$
    and 
    $$
    H_c^2(U_{\overline{k}},\mathcal{F}) = (\mathcal{F}_{\overline{\eta}})_{G^{\geom}}(-1),
    $$
    where $(\mathcal{F}_{\overline{\eta}})_{G^{\geom}}$ is the largest quotient module of $\mathcal{F}_{\overline{\eta}}$ on which $G^{\geom}$ acts trivially. In particular, if $\mathcal{F}$ is geometrically irreducible and nonconstant, then $H_c^2(U_{\overline{k}},\mathcal{F})=0.$
\end{lemma}

In our application of the Grothendieck--Lefschetz trace formula, the sheaves we consider are irreducible. In this case, Lemma~\ref{extremalCohomologyDescription} implies that we only require the dimension of the middle cohomology $H_c^1.$ The Grothendieck--Ogg--Shafarevich formula gives the alternating sum of the dimensions of the cohomology groups in terms of local data attached to the sheaf.

\begin{theorem}[Th. 4.2. of \cite{AppliedladicCohomology}]\label{EulerChar}
    If $\mathcal{F}$ is an $\ell$-adic sheaf lisse on $U\subset\A_{k}^1,$ then the Euler characteristic of $\mathcal{F}$ is given by
    $$
    \chi_c(U_{\overline{k}},\mathcal{F}):=\sum\limits_{i=0}^2(-1)^i\dim H_c^{i}(U_{\overline{k}},\mathcal{F}) = \rank(\mathcal{F})\cdot(2-\#(\mathbb{P}^1\setminus U)(\overline{k})) - \sum\limits_{\lambda\in(\mathbb{P}^1\setminus U)(\overline{k})}\Swan_\lambda(\mathcal{F}),
    $$
    where $\Swan_\lambda(\mathcal{F})$ is the Swan conductor of $\mathcal{F}$ at $\lambda.$ In particular, if $\mathcal{F}$ is irreducible and has at most tame ramification on $\mathbb{P}^1\setminus U,$ we have that
    $$
    \dim H_c^1(U_{\overline{k}},\mathcal{F})=\rank(\mathcal{F})\cdot (\#(\mathbb{P}^1\setminus U)(\overline{k})-2).
    $$
\end{theorem}

In order to prove our distribution results, we will show that the sheaves we consider are irreducible as representations of $G^{\geom}.$ To this end, we require the notion of geometric monodromy group. More precisely, if $\mathcal{F}$ is a lisse sheaf on $U$ associated to the representation $\rho_\mathcal{F},$ then the geometric monodromy group $G_{\geom}$ is the Zariski closure (for the definition and properties of Zariski closure, see Section 7.1 of \cite{KowalskiRep}) of $\rho_{\mathcal{F}}(G^{\geom})$ in $\GL(\mathcal{F}_{\overline{\eta}}).$ The sheaves we consider for the computation of the mixed moments arise as tensor products (in the representation-theoretic sense) of two sheaves. A special case of the Goursat--Kolchin--Ribet criterion allows us to compute the geometric monodromy groups of these tensor products.

\begin{theorem}[Th. 13.4 of \cite{AppliedladicCohomology}, 1.8.2 of \cite{KatzESDE}]\label{GKRCriterion}
    Let $\mathcal{F}_1$ and $\mathcal{F}_2$ be two generically rank-$2$ geometrically irreducible $\ell$-adic sheaves such that the following hold:
    \begin{enumerate}
        \item $\mathcal{F}_1$ and $\mathcal{F}_2$ are punctually pure of weight $0$ and self-dual (as representations).
        \item The geometric monodromy groups of $\mathcal{F}_1$ and $\mathcal{F}_2$ are $\SL_2.$
        \item For any rank one sheaf $\ell$-adic sheaf $\mathcal{L},$ we have that $\mathcal{F}_1$ is not geometrically isomorphic to $\mathcal{F}_2\otimes\mathcal{L},$ that is, not isomorphic as representations of $G^{\geom}$.
    \end{enumerate}
    Then, the geometric monodromy group of the sheaf $\mathcal{F}_1\bigoplus\mathcal{F}_2$ is $\SL_2\times\SL_2.$
\end{theorem}

Finally, note that in Lemma~\ref{hypergeometricClausen}, the relation between the $H_p(\lambda)$ function and the symmetric square of $a_p^{\Cl}(\lambda)$ involves a Legendre character $\phi_p.$ This meshes well with the framework of \'etale cohomology as the Kummer sheaf has traces equal to values of multiplicative characters. 

\begin{lemma}\cite[Section 3.4 and Example 4.2.2]{AppliedladicCohomology}\label{KummerSheafLemma}
    Let $p\neq \ell$ be a prime$ ,\chi:\F_p^\times\to\C^\times$ a character on $\F_p$ and $b\in\F_p.$ Then there exists a rank-$1$ $\ell$-adic sheaf $[+b]^\ast\mathcal{L}_{\chi}$, pure of weight $0$, lisse on $\mathbb{P}_{\F_p}^1\setminus\{b,\infty\},$ and tamely ramified at $b$ and $\infty,$ such that
    $$
    \tr(\Frob_\lambda|[+b]^\ast\mathcal{L}_\chi) = \chi(\lambda+b)
    $$
    for all $\lambda\in\mathbb{P}_{\F_p}^1\setminus\{b,\infty\}.$
\end{lemma}

\subsection{\'Etale sheaves arising from families of elliptic curves}\label{EtaleSheafFEC}

In his proof of equidistribution of traces of Frobenius for a generic family of elliptic curves \cite[Section 3]{RangMoyen}, Michel considers certain \'etale sheaves which contain information about those traces. Here, we recall the construction and properties of these sheaves for the reader's convenience.

Let $E_{\lambda}$ be a family of elliptic curves defined by the affine variety
$$
E_{\lambda}:\ \ \ y^2+a_1(\lambda)xy+a_3(\lambda)y=x^3+a_2(\lambda)x^2+a_4(\lambda)x+a_6(\lambda)\ \text{ with } \ a_1,\ldots,a_6\in\Z[\lambda],
$$
equipped with the projection map
$$
\begin{tikzcd}[row sep=-4pt, column sep= normal]
 \pi\colon\!\ \ \ \ \ \   E_\lambda \arrow[r] & \A_\Z^1 \\
\ \ \ \ \ \ \quad (x,y,\lambda) \arrow[r, maps to] & \lambda .   
\end{tikzcd}
$$
In this notation, we define $\mathcal{F}:=(R^1\pi_!\Q_\ell)(\frac{1}{2})$ to be the Tate twist of the first higher direct image with proper support of the constant sheaf on $E_\lambda$ (for the definition of higher direct images with proper support, see Chapter VI Section 3 of \cite{ECMilneBook}). 

We first recall some elementary properties of $\mathcal{F}.$

\begin{lemma}[Section 3 of \cite{RangMoyen}]\label{FLisse}
    Let $\Delta(\lambda)$ be the discriminant of $E_\lambda$ and let $\Delta^1(\lambda):=\frac{\Delta(\lambda)}{\gcd(\Delta(\lambda),\Delta'(\lambda))}.$ Furthermore, let $N\geq 1$ such that the leading coefficient of $\Delta^1(\lambda)$ is a unit in $\Z[\lambda,\frac{1}{N}].$ Then,  we have that $\rank(\mathcal{F})=2$ and that $\mathcal{F}$ is lisse on
    $$
    U_E=\Z\left[\lambda,\frac{1}{6\ell N\Delta^1(\lambda)}\right].
    $$
\end{lemma}

Now denote by $U_{p}$ and $\mathcal{F}_p$ the base-change of $U_E$ and $\mathcal{F}$ to $\F_p.$ Then, $\mathcal{F}_p$ is an $\ell$-adic sheaf, lisse on $U_p,$ pure of weight $0,$ and for $\lambda\in U_p,$ we have
\begin{equation}\label{traceF}
\tr(\Frob_\lambda|\mathcal{F}_{\overline{\eta}}) = \frac{a_{E,p}(\lambda)}{\sqrt{p}},
\end{equation}
where $a_{E,p}(\lambda) = p+1-\#E_\lambda(\F_p).$ In order to compute the moments and mixed moments, we require the geometric monodromy group, the self-duality, and the Euler characteristic of $\mathcal{F}.$

\begin{lemma}\label{FMonoChar}
    There exists $N\geq 1$ such that for all primes $p\nmid N,$ the following are true.
    \begin{enumerate}
        \item If the $j$-invariant $j(E_\lambda)$ is not constant as a function of $\lambda,$ then the geometric monodromy group of $\mathcal{F}_p$ is given by $G_{\geom}=\SL_2.$
        \item $\mathcal{F}_p$ is tamely ramified at the points of $\mathbb{P}^1_{\F_p}-U_p\times_{\Spec\F_p}\Spec\overline{\F_p}.$
        \item The sheaf $\mathcal{F}_p$ is self-dual.
    \end{enumerate}
\end{lemma}

\begin{proof}[Sketch of proof]
    Parts (1) and (2) are Lemma 3.1 of \cite{RangMoyen}. To see part (3), note that the smooth base change theorem (see Chapter VI, Corollary 4.2 of \cite{ECMilneBook}) implies that $(\mathcal{F}_p)_{\overline{\eta}}$ is isomorphic as a Galois representation to the $H_c^1(E_\lambda,\Q_\ell)\left(\frac{1}{2}\right),$ where $E_\lambda$ is viewed as an elliptic curve over $\F_p(\lambda).$ The self-duality then follows by Poincar\'e duality (see Chapter VI, Corollary 11.2 of \cite{ECMilneBook}) which in this case is nothing but the Weil pairing.
\end{proof}

Finally, in order to apply Theorem~\ref{GKRCriterion} in the case of generic pairs of families of elliptic curves, it is useful to compute the action of $I_\lambda$ on $\mathcal{F}_{\overline{\eta}},$ where $\lambda\in\A_{\F_p}$ is a point of bad reduction of $E_\lambda.$

\begin{lemma}[p. 135 of \cite{RangMoyen}, Section 4 of \cite{WeilII}]\label{InertiaAction}
    Let $\lambda\in\A_{\F_p}$ such that the fiber $E_\lambda$ has multiplicative reduction. Then, in some basis, the representation restricted to $I_\lambda$ is given by
    $$
    \rho_{\mathcal{F}}(\gamma) = \begin{pmatrix}
        1 & t_\ell(\gamma) \\
        0 & 1
    \end{pmatrix},
    $$
    where $t_\ell$ is the canonical surjection from $I_\lambda$ to $\Z_\ell(1).$
\end{lemma}

\section{Background from Representation Theory}\label{Section4}

Here we recall facts from representation theory which allow for computation of moments using the representations $\mathcal{F}_{\overline{\eta}}$ described in Section~\ref{Section3}. Aside from the technical details, this section is essentially a generalization of some of the tools used in \cite{HGFDistributionOnoSaikia,ono2024distributionhessianvaluesgaussian,grove2024hypergeometricmomentshecketrace}. More precisely, the results in this section combined with Lemma~\ref{FMonoChar} elucidate why the Chebyshev polynomial sums
$
\sum\limits_{\lambda}U_m\left(\frac{a_{p}(\lambda)}{2\sqrt{p}}\right)
$
are precisely the sums with small asymptotic magnitude. Furthermore, this section elucidates why the Catalan numbers are the moments and transforms the induction arguments on Chebyshev polynomials into direct explicit formulas.

In order to compute the moments of traces of Frobenius $a_p(\lambda),$ we make use of the tensor product representations of $\mathcal{F}_{\overline{\eta}}.$

\begin{definition}
    If $G$ is a group, $m\geq 0,$ and $\rho:G\to\GL(V)$ is a finite-dimensional representation of $G,$ the tensor product representation is given by
    $$
    \begin{tikzcd}[row sep=-4pt, column sep= normal]
     \rho^{\otimes m}\colon\!\ \ \ \ G \arrow[r] & GL(V^{\otimes m}) \\
    \ \ \ \ \ \ \ \quad g \arrow[r, maps to] & \rho^{\otimes m}(g),   
    \end{tikzcd}
    $$
    where
    $$
    \rho^{\otimes m}(g)(v_1\otimes\ldots\otimes v_m) := \rho(g)(v_1)\otimes\ldots\otimes\rho(g)(v_m).
    $$
\end{definition}
From the definition, it is obvious that $\Tr(\rho^{\otimes}(m))(g)=\Tr(\rho)(g)^m.$ In order to apply the results from Section~\ref{Section3} to $\rho^{\otimes m},$ where $\rho$ is the standard representation of $\SL_2,$ we decompose $\rho^{\otimes m}$ into a direct sum of irreducible representations.  To this end, recall that if $m\geq 0$ and $\rho:G\to\GL(V)$ is a finite-dimensional representation, then we have the $m$-th symmetric power representation
$$
\Sym^m\rho:G\to\GL(\Sym^mV),
$$
where $\Sym^m V:=V^{\otimes m}/S_m$ is the quotient of $V^{\otimes m}$ by the natural action of the symmetric group $S_m,$ and where $G$ acts naturally on this quotient space. In this notation, we have the following important facts when $\rho$ is the standard representation of $\SL_2.$

\begin{lemma}[Theorems 2.6.1 and 5.6.3 of \cite{KowalskiRep}]
    If $\rho:\SL_2\to V$ is the standard representation of $\SL_2$ on a vector space over a field of characteristic $0$ and $m\geq 0,$ then $\Sym^m\rho$ is an irreducible representation of $\SL_2$ and for all $g\in G,$ we have
    $$
    \tr(g|\Sym^m V) = U_m\left(\frac{\tr(g|V)}{2}\right),
    $$
    where $U_m$ is the $m$-th Chebyshev polynomial of the second kind.
\end{lemma}

In order to compute the moments in Theorems~\ref{theorem1} and \ref{theorem3}, we decompose the tensor product representations $\rho^{\otimes m}$ into a direct sum of symmetric power representations and a trivial representation. The aim of this is to isolate the irreducible nontrivial components, which by Lemma~\ref{extremalCohomologyDescription}, Theorem~\ref{GrothendieckLefschetz}, and Theorem~\ref{DelignePurity} have small contributions to the moments. The following well-known lemma (for example, see 2.2. of \cite{TemperleyLieb}) provides the required direct sum decomposition.

\begin{lemma}\label{oneRepMult}
    If $\rho$ is the standard representation of $\SL_2$ and $m\geq 0,$ then we have
    $$
    \rho^{\otimes m}\cong\bigoplus\limits_{\substack{r\leq m \\ r\equiv m\Mod{2}}} n_m(r)\Sym^r\rho, 
    $$
    where
    $$
    n_m(r) = \binom{m}{\frac{m+r}{2}}\cdot\frac{2(r+1)}{m+r+2}.
    $$
    In particular, when $m=2m_1$ is even, the multiplicity $n_m(0)$ of the trivial representation is given by the Catalan number $C(m_1).$
\end{lemma}

\begin{remark}
We give the following two remarks on Lemma~\ref{oneRepMult}.
\begin{enumerate}
\item In \cite{HGFDistributionOnoSaikia, ono2024distributionhessianvaluesgaussian,grove2024hypergeometricmomentshecketrace}, the authors compute the moments by an induction argument on the contributions of the Chebyshev polynomials of the traces of Frobenius. Lemma~\ref{oneRepMult} serves as a direct expression of the $m$-th moments in terms of the various Chebyshev polynomials that appear in those works. For example, if $m=4,$ then we have
\begin{align*}
    \sum\limits_{\lambda\in\F_p}\left(\frac{a_p(\lambda)}{\sqrt{p}}\right)^4 &=\sum\limits_{\lambda\in\F_p} n_4(0)U_0\left(\frac{a_p(\lambda)}{2\sqrt{p}}\right) + n_4(2) \sum_{\l \in \fp} U_2\left(\frac{a_p(\lambda)}{2\sqrt{p}}\right) + n_4(4) \sum_{\l \in \fp}U_4\left(\frac{a_p(\lambda)}{2\sqrt{p}}\right) \\ 
    &= 2p + 3\sum\limits_{\lambda\in\F_p}U_2\left(\frac{a_p(\lambda)}{2\sqrt{p}}\right) + \sum\limits_{\lambda\in\F_p}U_4\left(\frac{a_p(\lambda)}{2\sqrt{p}}\right).
    \end{align*}
This, for example, makes it clear why the statement that 
$$
\lim\limits_{p\to\infty}\frac{1}{p}\sum\limits_{\lambda\in\F_p} U_m\left(\frac{a_p(\lambda)}{2\sqrt{p}}\right) = 0
$$
for all $m>0,$ proved in the aforementioned works, implies that the $4$-th moment is given by the Catalan number $C(2)=2.$

\item In our work, we only make use of the multiplicity of the trivial representation in $\rho^{\otimes m}.$ The full decomposition is noted only for completeness.
\end{enumerate}
\end{remark}

In the setting of Theorem~\ref{theorem4}, Lemma~\ref{hypergeometricClausen} implies that if $m\geq 0$ is even, then the $m$-th moments of the $H_p$ functions are normalized even moments of $U_2\left(\frac{a^{\Cl}_p(\lambda)}{2\sqrt{p}}\right).$ In order to compute these moments, we decompose the tensor product representation of $(\Sym^2\rho)^{\otimes m}$ into a direct sum of symmetric power representations and a trivial representation. For the even moments, we require the following lemma which gives the multiplicity of the trivial representation.

\begin{lemma}\label{symsquareRepMult}
    If $\rho$ is the standard representation of $\SL_2$ and $m\geq 0$, then the multiplicity of the trivial representation in the direct sum decomposition of $(\Sym^2\rho)^{\otimes m}$ is given by
    $$
    (-1)^m\cdot\sum\limits_{i=0}^m(-1)^i\binom{m}{i}\frac{(2i)!}{i!(i+1)!}.
    $$
\end{lemma}

\begin{proof}
If $m=0,$ then this is trivial. Now suppose $m\geq1.$
Using Lemma~\ref{oneRepMult}, we have that
$$
\rho^{\otimes 2}\cong \Sym^2\rho\oplus\varepsilon,
$$
where $\varepsilon$ is the trivial representation. Taking $m$-th tensor powers on both sides and comparing the multiplicities of the trivial representation, we have that
$$
\frac{(2m)!}{m!(m+1)!} = \sum\limits_{i=0}^m \binom{m}{i}n_\varepsilon((\Sym^2\rho)^{\otimes i}),
$$
where $n_\varepsilon((\Sym^2\rho)^{\otimes i})$ is the multiplicity of the trivial representation in the direct sum decomposition of $(\Sym^2\rho)^{\otimes i}.$ Therefore, if we define
$$
a(m):=(-1)^m\cdot\sum\limits_{i=0}^m(-1)^i\binom{m}{i}\frac{(2i)!}{i!(i+1)!},
$$
then, by induction on $m,$ our claim is equivalent to showing that
$$
\sum\limits_{i=0}^{m}\binom{m}{i}a(i) = \frac{(2m)!}{m!(m+1)!}.
$$
Expanding the definition of $a(i),$ this is equivalent to showing that
$$
\sum\limits_{i=0}^{m-1}(-1)^i\sum\limits_{j=0}^i(-1)^j\binom{i}{j}\binom{m}{i}\cdot\frac
{(2j)!}{j!(j+1)!} = -\sum\limits_{i=0}^{m-1}(-1)^{m+i}\binom{m}{i}\frac{(2i)!}{i!(i+1)!}.
$$
By switching the order of summation, we write the left-hand side as
$$
\sum\limits_{j=0}^{m-1}\sum\limits_{i=j}^{m-1}(-1)^{i+j}\binom{i}{j}\binom{m}{i}\frac{(2j)!}{j!(j+1)!}.
$$
However, we have that the inner sum is given by
$$
\frac{(2j)!}{j!(j+1)!}\sum\limits_{i=0}^{m-j-1}(-1)^i\frac{m!}{j!i!(m-j-i)!},
$$
which by the binomial theorem is clearly equal to
$$
(-1)^{j+m+1}\cdot\binom{m}{j}\cdot\frac{(2j)!}{j!(j+1)!}.
$$
Putting all this together, we have our claim.
\end{proof}

For the odd moments, we require the moments of $U_2\left(\frac{a_p^{\Cl}(\lambda)}{2\sqrt{p}}\right)$ twisted by the trace $\phi_p(\lambda+1)$ of a Kummer sheaf. To compute those moments, we require the following elementary lemma (see Exercise 2.2.14 of \cite{KowalskiRep} for example) on the irreducibility of tensor products with one-dimensional representations.

\begin{lemma}\label{twistIrreducibility}
    If $G$ is a group, $V$ a finite-dimensional vector space over a field $k,$  $\rho:G\to\GL(V)$ an irreducible representation of $G,$ and $\chi:G\to k^\times$ a one-dimensional representation of $G,$ then $\rho\otimes\chi$ is irreducible.
\end{lemma}

In order to use the geometric monodromy computation of Deligne in Lemma~\ref{FMonoChar}, we connect the irreducibility of $G^{\geom}$ with $G_{\geom}.$
Recall that the geometric mondromy group $G_{\geom}$ is the Zariski closure of $G^{\geom}$ in $\GL(\mathcal{F}_{\overline{\eta}}).$ The following well-known lemma (for example, see Exercise 7.1.6 of \cite{KowalskiRep}) relates the irreducibility of a representation
$$
\rho: G\to\GL(V)
$$
to the irreducibility of the (standard) representation of its Zariski closure $\overline{\rho(G)}$ in $\GL(V).$

\begin{lemma}\label{RepZar}
    Let $G$ be a group and $\rho:G\to\GL(V)$ be a finite-dimensional linear representation of $G.$ We have that $\rho$ is irreducible if and only if the standard representation of $\overline{\rho(G)}$ is irreducible.
\end{lemma}

\begin{proof}
    Here we write the proof for the reader's convenience.
    If $\rho$ is irreducible, it follows immediately that $\overline{\rho(G)}$ acts irreducibly. Now, suppose the standard representation of $\overline{\rho(G)}$ is irreducible, and let $W$ be a subspace of $V$ such that 
    $$
    \rho(G)(W)\subset W.
    $$
    Let $\mathbf{H}$ be the subset of elements $\mathbf{h}\in\overline{\rho(G)}$ such that
    $$
    \mathbf{h}\cdot W\subset W.
    $$
    It is clear that $\mathbf{H}$ is a Zariski closed algebraic subgroup of $\GL(V)$ which contains $\rho(G).$ This implies that $\mathbf{H}=\overline{\rho(G)}.$ Since $\overline{\rho(G)}$ acts irreducibly, $W=\{0\}$ or $W=V,$ proving our claim.
\end{proof}

The following technical lemma plays an essential role in deducing irreducible decomposition of tensor products of sheaves from the geometric monodromy computation in Theorem~\ref{GKRCriterion}.

\begin{lemma}\label{geoMonTens}
    Let $G$ be a group, $\rho_1,\rho_2$ be finite-dimensional representations of $G,$ and let $r_1,r_2$ be Zariski-continuous irreducible representations of the Zariski closures $\overline{\rho_1(G)},\overline{\rho_2(G)}$ respectively. If 
    $$
    \overline{(\rho_1\oplus\rho_2)(G)}=\overline{\rho_1(G)}\times\overline{\rho_2(G)},
    $$
    then $r_1\circ\rho_1\otimes r_2\circ\rho_2$ is an irreducible representation of $G.$
\end{lemma}

\begin{proof}
    The assumption
    $$
    \overline{(\rho_1\oplus\rho_2)(G)}=\overline{\rho_1(G)}\times\overline{\rho_2(G)},
    $$ 
    implies that
    \begin{equation}\label{1Imp}
    (r_1,r_2)(\overline{(\rho_1\oplus\rho_2)(G)})=r_1(\overline{\rho_1(G)})\times r_2(\overline{\rho_2(G)}).
    \end{equation}
    Since the pair $(r_1,r_2)$ is Zariski-continuous, we have that
    \begin{equation}\label{2Imp}
    \overline{(r_1,r_2)((\rho_1\oplus\rho_2)(G))}\supset (r_1,r_2)(\overline{(\rho_1\oplus\rho_2)(G))}.
    \end{equation}
    Putting (\ref{1Imp}) and (\ref{2Imp}) together, we have that
    $$
    \overline{(r_1\circ\rho_1\otimes r_2\circ\rho_2)(G)}\supset(r_1\boxtimes r_2)(\overline{\rho_1(G)}\times\overline{\rho_2(G)}),
    $$
    where $r_1\boxtimes r_2$ is the external tensor product (see p. 64 of \cite{KowalskiRep}) of $r_1$ and $r_2.$ Since $r_1$ and $r_2$ are irreducible, we have that (see Proposition 2.3.23 of \cite{KowalskiRep}) $r_1\boxtimes r_2$ is irreducible, and therefore, the standard representation of $\overline{(r_1\circ\rho_2\otimes r_2\circ\rho_2)(G)}$ is irreducible. By Lemma~\ref{RepZar}, we conclude that $r_1\circ\rho_1\otimes r_2\circ\rho_2$ is an irreducible representation of $G.$
    
\end{proof}

Finally, in order to verify the conditions of Theorem~\ref{GKRCriterion} in our setting by applying Lemma~\ref{InertiaAction}, we compute the space of invariants after twisting by a rank-$1$ $\ell$-adic sheaf.

\begin{lemma}\label{inertGeoConst}
    Let $I$ be a group and $\rho:I\to\GL(V)$ be a rank-$2$ representation of $I$ such that under some basis, for every $\gamma\in I,$ we have
    $$
    \rho(\gamma)=\begin{pmatrix}
        1 & t(\gamma) \\
        0 & 1
    \end{pmatrix},
    $$
    where $t(\gamma)\neq 0$ for all $\gamma\in I.$ In this notation, if $\chi$ is a nontrivial rank-$1$ representation of $I,$ then the space of invariants of $\rho\otimes\chi$ is $\{0\}.$
\end{lemma}

\begin{proof}
    Let $\{v_1,v_2\}$ be the basis in the statement of the lemma, and let $v=c_1v_1+c_2v_2$ be an element of $V.$ If the subspace generated by $v$ is invariant under the action of $(\rho\otimes\chi)(I),$ a direct computation shows that $c_1=c_2=0.$
\end{proof}

\section{Moments of Frobenius Traces}\label{Section5}

Here we assemble the results from Sections~\ref{Section3} and \ref{Section4} to compute the moments and mixed-moments that are required for Theorems~\ref{theorem1} through \ref{theorem4}.

To make this precise, for $i=1,2,$ let $E_{i,\lambda}$ be the two affine varieties over $\A^1_\Z$ given by
$$
E_{i,\lambda}:\ \ \ y^2+a_{i,1}(\lambda)xy+a_{i,3}(\lambda)y=x^3+a_{i,2}(\lambda)x^2+a_{i,4}(\lambda)x+a_{i,6}(\lambda)
$$
with $a_{i,1}(\lambda)\ldots a_{i,6}(\lambda)\in\Z[\lambda].$ Furthermore denote by $\pi_i:E_{i,\lambda}\to \A^1_\Z$ the projection map
$$
\pi_i(x,y,\lambda)=\lambda
$$
and define $\mathcal{F}_i:=(R^1\pi_{i,!}\Q_\ell)(\frac{1}{2}).$ For a prime $p,$ we denote by $\mathcal{F}_{i,p}$ the base-change of $\mathcal{F}_i$ to $\A^1_{\F_p}.$  We further assume that there the $j$-invariants of $E_{1,\lambda}$ and $E_{2,\lambda}$ are nonconstant, and that there exists $N\geq 1$ such that for all $p\nmid N$ and all rank-$1$ $\ell$-adic sheaves $\mathcal{L}_p$ on $\A_{\F_p}^1,$ we have
$$
\mathcal{F}_{1,p}\not\cong_{\text{geom}}\mathcal{F}_{2,p}\otimes\mathcal{L}_p.
$$
Finally, write $a_{i,p}(\lambda):=p+1-\#E_{i,\lambda}(\F_p).$ In this notation, we have the following preliminary version of Theorem~\ref{theorem2}.

\begin{proposition}\label{mixedMoments}
If $m,n\geq 0,$ then we have
$$
\lim\limits_{p\to\infty}\frac{1}{p^{m/2+n/2+1}}\sum\limits_{\lambda\in\F_p}a_{1,p}(\lambda)^m a_{2,p}(\lambda)^n = \begin{cases}
    C(m_1)C(n_1) &\ \ \ \text{ if } m=2m_1\text{ and }n=2n_1\text{ are even} \\ 
    0 &\ \ \ \text{ otherwise}.
\end{cases}
$$
\end{proposition}

Note that the case $n=0$ computes the moments of one-parameter families of elliptic curves with nonconstant $j$-invariant (without need of choice for $E_{2,\lambda},$ by the same proof), recovering the result of Michel \cite[Section 3]{RangMoyen}, of which this is a generalization.

In order to prove this proposition, we first compute the multiplicity of the trivial representation in the $\ell$-adic sheaves $\mathcal{G}_{p,m,n}:=\mathcal{F}_{1,p}^{\otimes m}\otimes\mathcal{F}_{2,p}^{\otimes n}.$ To make this precise, let $U$ be the subset of $\A_{\Z}^1$ on which both $\mathcal{F}_1$ and $\mathcal{F}_2$ are lisse (see Lemma~\ref{FLisse}) and for each prime $p,$ denote by $U_p$ and $\overline{U_p}$ the base-changes of $U$ to $\F_p$ and $\overline{\F_p}$ respectively. In this notation, we have the following computation.

\begin{lemma}\label{multTwoRep}
    If $m,n\geq 0$ and $p\nmid N,$ then the multiplicity of the trivial representation in the direct-sum decomposition of $\mathcal{G}_{p,m,n}$ as a representation of $G^{\geom}$ is given by
    $$
    n_{\varepsilon}(\mathcal{G}_{m,n,p}) = \begin{cases}
    C(m_1)C(n_1) &\ \ \ \text{ if } m=2m_1\text{ and }n=2n_1\text{ are even} \\ 
    0 &\ \ \ \text{ otherwise}.
    \end{cases}
    $$
\end{lemma}

\begin{proof}
    Denote by $\rho_{i,p}$ the representation
    $$
    \rho_{i,p}:G^{\geom}\to\mathcal{F}_{i,p,\overline{\eta}}.
    $$
    As representations of $G^{\geom},$ if $i=1,2,$  we have the decompositions
    \begin{equation}\label{decompEq}
    \rho_{i,p}^{\otimes m} \cong \left(\bigoplus\limits_{k=1}^{n_i}n_m(i,k)\cdot \pi_{i,k,p}\circ\rho_{i,p}\right)\oplus n_m(i,0)\varepsilon,
    \end{equation}
    where the representations $\pi_{i,k}$ are irreducible pairwise nonisomorphic nontrivial representations of $\rho_{i,p}(G^{\geom})$ and where $\varepsilon$ is the trivial representation.  Part (1) of Lemma~\ref{FMonoChar} states that the Zariski closure of $\rho_{i,p}(G^{\geom})=\SL_2.$ Therefore, by Lemma~\ref{RepZar}, we have that the $\pi_{i,k,p}$ are precisely the symmetric power representations and $n_m(i,k)=n_{m}(k)$ are the multiplicities in Lemma~\ref{oneRepMult}. 
    Now, note that by Theorem~\ref{GKRCriterion}, we have that the Zariski closure of $(\rho_{1,p}\oplus\rho_{2,p})(G^{\geom})$ is given by $\SL_2\times\SL_2.$ Therefore, Lemma~\ref{geoMonTens} implies that for $1\leq k_1\leq n_1, 1\leq k_2\leq n_2,$ we have that
    $
    \pi_{1,k_1,p}\circ\rho_{1,p}\otimes \pi_{1,k_2,p}\circ\rho_{2,p}
    $
    is an irreducible representation of $G^{\geom}.$ It is also clear that $\pi_{1,k_1,p}\otimes \varepsilon$ and $\varepsilon\otimes\pi_{1,k_2,p}$ are irreducible.
    Therefore, the multiplicity of the trivial representation in the decomposition of $\mathcal{G}_{m,n,p}$ is equal to $n_m(1,0)n_m(2,0)$ and then Lemma~\ref{oneRepMult} shows our claim.
\end{proof}

The next lemma computes the asymptotics of trace sums of the nontrivial irreducible representations arising in the direct sum decomposition of $\mathcal{G}_{p,m,n}.$ More explicitly, the proof of Lemma~\ref{multTwoRep} shows that these representations are of the form $\Sym^{m_1}(\mathcal{F}_1)\otimes\Sym^{n_1}(\mathcal{F}_2),$ where $m_1\leq m,n_1\leq n,$ with appropriate multiplicities. Therefore, the following lemma proves that if $m_1$ and $n_1$ are not both zero, then we have
$$
\lim\limits_{p\to\infty}\frac{1}{p}\sum\limits_{\lambda\in\F_p} U_{m_1}\left(\frac{a_{1,p}(\lambda)}{2\sqrt{p}}\right)U_{n_1}\left(\frac{a_{2,p}(\lambda)}{2\sqrt{p}}\right)=0.
$$

\begin{lemma}\label{mixedMomNonTrivial}
    If $p$ is a prime and $m,n\geq 0,$ then we have
    $$
    \sum\limits_{\pi}\sum\limits_{\lambda\in U_p(\F_p)}\tr(\Frob_\lambda|\pi) = O_{m,n}(p^{1/2})\ \ \text{ as }p\to\infty,
    $$
   where $\pi$ runs over the nontrivial irreducible representations arising in the direct sum decomposition of $\mathcal{G}_{p,m,n}.$
\end{lemma}

\begin{proof}
    If $\pi$ is an irreducible representation as in the statement, it is clear that $\pi,$ viewed as an $\ell$-adic sheaf, is lisse on $U_p.$ Therefore, Theorem~\ref{GrothendieckLefschetz} gives us that
    $$
    \sum\limits_{\lambda\in U_p(\F_p)}\tr(\Frob_\lambda|\pi)=\sum\limits_{i=0}^2\tr(\Frob_p|H_c^i(\overline{U_p},\pi)).
    $$
    By Lemma~\ref{extremalCohomologyDescription}, we have that $H_c^0=H_c^2=0,$ and therefore,
    $$
    \sum\limits_{\lambda\in U_p}\tr(\Frob_\lambda|\pi)=\tr(\Frob_p|H_c^1(\overline{U_p},\pi)).
    $$
    Theorems~\ref{DelignePurity}, \ref{EulerChar} and Lemma~\ref{FMonoChar} then imply that
    $$
    \left|\sum\limits_{\lambda\in U_p}\tr(\Frob_\lambda|\pi)\right|\leq p^{1/2}\cdot\rank(\pi)\cdot(\#(\mathbb{P}^1\setminus\overline{U_p})(\overline{\F_p})-2).
    $$
    Since $\sum\limits_{\pi}\rank(\pi)\leq\rank(\mathcal{G}_{p,m,n})=2^{m+n},$ the claim follows.
\end{proof}

\begin{remark}\label{Hc1ModRemark}
    In the proof of Lemma~\ref{mixedMomNonTrivial}, we show that as $p\to\infty,$ we have
    $$
    \sum\limits_{\lambda\in\F_p} U_{m_1}\left(\frac{a_{1,p}(\lambda)}{2\sqrt{p}}\right)U_{n_1}\left(\frac{a_{2,p}(\lambda)}{2\sqrt{p}}\right) = O(p^{1/2})
    $$
    by identifying the sum with the trace of Frobenius on the cohomology group $H_c^1(\overline{U_p},\pi),$ where $\pi=\Sym^{m_1}(\mathcal{F}_1)\otimes\Sym^{n_1}(\mathcal{F}_2).$
    Now, for $d\in\{2,3,4\},$ let $n_1=0$ and let $E_{1,\lambda}$ be the family $\widetilde{E}_{d,\lambda}$ defined in Section~\ref{Section2}. In this case, we have an explicit description of the Frobenius traces on this ``mysterious'' cohomology group in terms of the traces of Hecke operators on spaces of cusp frorms. For example, if $d=3$ and $m_1=2r$ is even, then (see Section 4.2 of \cite{grove2024hypergeometricmomentshecketrace}) the trace of the Frobenius map is given by 
    $$
    \tr(\Frob_p|H_c^1(\overline{U_p},\pi))=-\frac{\Tr_{2r+2}(\Gamma_1(3),p)-c(2r)}{p^r},
    $$
    where $\Tr_{\nu}(\Gamma,p)$ is the trace of the $p$-th Hecke operator on the space of cusp forms of weight $\nu$ over $\Gamma,$ and where $c(2r)$ is an explicit expression in terms of $p.$  The same remark applies to odd $m_1$ and to the elliptic curves $\widetilde{E}_{d,\lambda}.$ The advantage of the modular method is that one understands the trace of Frobenius on the cohomology very explicitly. On the other hand, the advantage of the machinery of \'etale cohomology is that one can derive results on distributions in a more general setting where there is no underlying modular framework.
\end{remark}

Now we prove Proposition~\ref{mixedMoments}.

\begin{proof}[Proof of Proposition~\ref{mixedMoments}]
By (\ref{traceF}), we have that
$$
\frac{1}{p^{m/2+n/2+1}}\sum\limits_{\lambda\in \F_p}a_{1,p}(\lambda)^m a_{2,p}(\lambda)^n=\frac{1}{p}\sum\limits_{\lambda\in U_p(\F_p)}\tr(\Frob_\lambda|\mathcal{G}_{p,m,n})+O_{m,n}(p^{-1}),
$$
as $p\to\infty.$ But then, Lemma~\ref{mixedMomNonTrivial} implies that
$$
\frac{1}{p}\sum\limits_{\lambda\in U_p(\F_p)}\tr(\Frob_\lambda|\mathcal{G}_{p,m,n}) = n_{\varepsilon}(\mathcal{G}_{m,n,p})+O_{m,n}(p^{-1/2}),
$$
where $n_{\varepsilon}$ is the multiplicity of the trivial representation in $\mathcal{G}_{m,n,p}.$ Combining this with Lemma~\ref{multTwoRep}, we obtain our claim.
\end{proof}

In order to obtain the necessary moments for Theorem~\ref{theorem4}, we require the following modification of Proposition~\ref{mixedMoments}.

\begin{proposition}\label{3F2MomEC}
    If $m\geq 0,$ then we have
    $$
    \lim\limits_{p\to\infty}\frac{1}{p^{m+1}}\sum\limits_{\lambda\in\F_p}(\phi_p(\lambda+1)(a_p^{\Cl}(\lambda)^2-p))^m=\begin{cases}
        \sum\limits_{i=0}^m (-1)^i\binom{m}{i}\frac{(2i)!}{i!(i+1)!} &\ \ \ \text{if }m\text{ is even} \\
        0 &\ \ \ \text{if }m\text{ is odd.}
    \end{cases}
    $$
\end{proposition}

\begin{proof}
    If $m$ is even, then we have that
    $$
    \frac{1}{p^m}\sum\limits_{\lambda\in\F_p}(\phi_p(\lambda+1)(a_p^{\Cl}(\lambda)^2-p))^m = \sum\limits_{\lambda\in U_p(\F_p)}\tr(\Frob_\lambda|(\Sym^2\mathcal{F})^{\otimes m}) + O_m(p^{-1}).
    $$
    Using Lemma~\ref{multTwoRep}, we decompose the right-hand sum as
    \begin{align*}
    \sum\limits_{\lambda\in U_p(\F_p)}\tr(\Frob_\lambda|(\Sym^2\mathcal{F})^{\otimes m}) &= \sum\limits_{\lambda\in U_p(\F_p)}\sum\limits_{\pi}n_2(\pi)\cdot\tr(\Frob_\lambda
    |\pi) \\ 
    &+\left(\sum\limits_{i=0}^m (-1)^i\binom{m}{i}\frac{(2i)!}{i!(i+1)!} \right)(p-2),
    \end{align*}
    where $\mathcal{F}$ is the sheaf associated to the Clausen curves $E_\lambda^{\Cl}$ as in Section~\ref{EtaleSheafFEC} and where $\pi$ ranges over $\Sym^r\mathcal{F}$ with $r\leq m$ and $r\equiv 0\pmod 2.$ Using exactly the same argument as in the proof of Lemma~\ref{mixedMomNonTrivial}, we obtain our claim.
    
    If $m$ is odd, note that
    $$
    \frac{1}{p^m}\sum\limits_{\lambda\in\F_p}(\phi_p(\lambda+1)(a_p^{\Cl}(\lambda)^2-p))^m = \sum\limits_{\lambda\in U_p(\F_p)}\tr(\Frob_\lambda|(\Sym^2\mathcal{F})^{\otimes m}\otimes ([+1]^\ast\mathcal{L}_\phi) + O_m(p^{-1}),
    $$
    where $\mathcal{L}_\phi$ is the Kummer sheaf associated to the Legendre character $\phi_p.$ Decomposing as in the case for even moments and using the orthogonality of characters, we have that
    $$
    \frac{1}{p^m}\sum\limits_{\lambda\in\F_p}(\phi_p(\lambda+1)(a_p^{\Cl}(\lambda)^2-p))^m = \sum\limits_{\lambda\in U_p}\sum\limits_{\pi}n_2(\pi)\cdot\tr(\Frob_\lambda|\pi\otimes[+1]^\ast\mathcal{L}_\phi) + O_m(1).
    $$
    Now, note that these irreducible representations $\pi$ are symmetric powers of $\mathcal{F}.$ By Lemma~\ref{twistIrreducibility}, we have that $\pi\otimes [+b]^\ast\mathcal{L}_\phi$ is irreducible. Moreover, it is clear that this representation is never trivial. Therefore, the same argument as in the proof of Lemma~\ref{mixedMomNonTrivial} shows our claim.
\end{proof}

\begin{remark}
    Just as in Proposition~\ref{mixedMoments}, the cohomology groups that show up in the proof of Proposition~\ref{3F2MomEC} have explicit traces using modular forms. For more details, see Proposition 2.2 and the proof of Proposition 6.3 in \cite{3F2Explicit}.
\end{remark}

Finally, to compute the moments in Theorem~\ref{theorem1} in terms of the mixed moments in Theorem~\ref{theorem2}, we require the following combinatorial lemma.

\begin{lemma}\label{CombMom}
    If $m\geq 1,$ then we have
    $$
    C(m)C(m+1) = \sum_{s=0}^{m} \binom{2m}{2s} C(m-s)C(s).
    $$
\end{lemma}

\begin{proof}
If $s\geq 0,$ by definition of the Catalan numbers, we have
$$
C(s) = \frac{1}{s+1} \binom{2s}{s}.
$$
Expanding the right-hand side, we then have that
$$
\sum_{s=0}^{m} \binom{2m}{2s} C(m-s)C(s) = \sum_{s=0}^{m} \frac{1}{(s+1)(m-s+1)} \frac{(2m)!}{(s!)^{2}((m-s)!)^{2}}.
$$
We rewrite this expression as
$$
\sum_{s=0}^{m} \binom{2m}{2s} C(m-s)C(s) = \frac{(2m)!}{m!(m+2)!} \sum_{s=0}^{m} \binom{m+2}{s+1} \binom{m}{m-s}.
$$
An application of the Chu--Vandermonde identity (see Example 1.17 of \cite{EnumCombinatorics}) then gives that
$$
\sum_{s=0}^{m} \binom{2m}{2s} C(m-s)C(s) = \frac{(2m)!}{m!(m+1)!} \cdot \frac{(2m+2)!}{(m+1)!(m+2)!}
$$
which proves our claim.
\end{proof}

\section{Tate Modules and Isogenies}\label{Section6}

Here we discuss the connection between the (untwisted) sheaves $\mathcal{F}:=R^1\pi_!\Q_\ell$ discussed in Section~\ref{EtaleSheafFEC} and $\ell$-adic Tate modules, and then express the non-isomorphism condition required for Proposition~\ref{mixedMoments} in terms of isogenies.

First, if $K$ is a field, $E/K$ is an elliptic curve, and $\ell\neq\text{char}(K)$ is prime, recall that the rational $\ell$-adic Tate module is defined by
$$
V_\ell(E):=(\varprojlim E[\ell^n])\otimes_{\Z_\ell}\Q_\ell
$$
and that $V_\ell(E)$ admits a natural action of the underlying Galois group $\Gal(\overline{K}/K)$ (see III.7 of \cite{SilvermanAEC}).  It turns out that the sheaves we consider ``are'' nothing but duals of those rational Tate modules when $E_\lambda$ is viewed as an elliptic curve over $\F_p(\lambda).$

\begin{lemma}\label{TateSheaf}
    If $E_\lambda$ and $\mathcal{F}$ are as in Section~\ref{EtaleSheafFEC}, then
    $$
    \mathcal{F}_{\overline{\eta}}\cong V_\ell(E_\lambda)^\vee,
    $$
    where $E_\lambda$ is viewed as an elliptic curve over $\F_p(\lambda)$ and where the isomorphism is an isomorphism of representations of the Galois group $\Gal(\F_p(\lambda)^{\sep}/\F_p(\lambda)).$
\end{lemma}

\begin{proof}
    This is Theorem 15.1 of \cite{MilneAVCollection} and Corollary 4.2 of \cite{ECMilneBook}.
\end{proof}

The following result of Paršin plays an essential role in showing that a generic pair of elliptic curves satisfies the conditions required to use Proposition~\ref{mixedMoments}.

\begin{theorem}[Corollary 1 of Theorem 2 of \cite{Parsin}]\label{Parsin}
    Let $\ell$ be a prime and let $K$ be a function field of transcendence degree $1$ over $\F_p$ where $p\neq 2,\ell.$ If $E_{1,\lambda}$ and $E_{2,\lambda}$ are elliptic curves over $K$ and $V_\ell(E_{1,\lambda})\cong V_\ell(E_{2,\lambda}),$ then $E_{1,\lambda}$ and $E_{2,\lambda}$ are isogenous over $K.$
\end{theorem}

Now, let $E_{1,\lambda},E_{2,\lambda},\mathcal{F}_1,\mathcal{F}_2$ be as in Section~\ref{Section5}. The following proposition describes the ``independence'' condition of a generic pair in terms of isogenies of elliptic curves.

\begin{proposition}\label{IsogenyCondition}
    Let $(E_{1,\lambda},E_{2,\lambda})$ be a generic pair. There exists $N\geq 1$ such that if $p\nmid N$ and $\mathcal{L}_p$ is a rank-$1$ $\ell$-adic sheaf on $\mathbb{P}_{\F_p}^1,$ then we have
    $$
    \mathcal{F}_{1,p}\not\cong_{\geom}\mathcal{F}_{2,p}\otimes\mathcal{L}_p.
    $$
\end{proposition}

In order to prove Proposition~\ref{IsogenyCondition}, we require the following two lemmas.

\begin{lemma}\label{geomConstant}
    Assume $E_{1,\lambda}$ and $E_{2,\lambda}$ have either good or multiplicative reduction everywhere. There exists $N\geq 1$ such that if $p\nmid N$ is a prime and $\mathcal{L}_p$ is a rank-$1$ $\ell$-adic sheaf on $\mathbb{P}_{\F_p}^1$ with
    $$
    \mathcal{F}_{1,p}\cong_{\geom}\mathcal{F}_{2,p}\otimes\mathcal{L}_p,
    $$
    then $\mathcal{L}_p$ is geometrically constant.
\end{lemma}

\begin{proof}
    First, by Lemma~\ref{InertiaAction} and our assumption of good or multiplicative reduction, it is clear that $\mathcal{F}_{1,p}$ and $\mathcal{F}_{2,p}$ must have the same points of ramification. By Lemma~\ref{inertGeoConst}, we have that $\mathcal{L}_p$ must be lisse on $\A_{\F_p}^1.$ Furthermore, by Lemma~\ref{FMonoChar}, there exists $N\geq 1$ such that if $p\nmid N$, then $\mathcal{F}_{1,p}$ and $\mathcal{F}_{2,p}$ are tamely ramified at $\infty$ and therefore, $\mathcal{L}_p$ is tamely ramified at $\infty.$ Therefore, $\mathcal{L}_p$ is geometrically constant (see Example 5.2 (f) of Chapter I of \cite{ECMilneBook}).
\end{proof}

\begin{lemma}\label{geomConstantQuad}
    Assume that $E_{1,\lambda}$ and $E_{2,\lambda}$ have either good or multiplicative reduction everywhere. There exists $N\geq 1$ such that if $p\nmid N$ is a prime and $\mathcal{L}_p$ is a rank-$1$ $\ell$-adic sheaf on $\mathbb{P}_{\F_p}^1$ with
    $$
    \mathcal{F}_{1,p}\cong_{\geom}\mathcal{F}_{2,p}\otimes\mathcal{L}_p,    
    $$
    then we have
    $$
    \mathcal{F}_{2,p}\otimes\mathcal{L}_p\cong\mathcal{F}_{2,d,p},
    $$
    where $d\in\F_p^\times, \pi_{2,d}$ is the projection map on the quadratic twist $E_{2,\lambda,d}$ and $\mathcal{F}_{2,d,p}=R^1\pi_{2,d,!}\Q_\ell.$
\end{lemma}

\begin{proof}
    If $\mathcal{L}_p$ is geometrically constant, then it follows by Proposition 2.8.2 of \cite{KowalskiRep} and (\ref{ArithGeoExactSequence}) that for $\lambda\in\A_{\F_p}^1,$ we have
    $$
    \tr(\Frob_\lambda|\mathcal{L}_p)=\chi(\Frob_p^{\deg \lambda}),
    $$
    where $\chi$ is a character on $\Gal(\overline{\F_p}/\F_p).$ By observing that the traces of Frobenius $\mathcal{F}_{1,p}$ and $\mathcal{F}_{2,p}$ are $\R$-valued for all embeddings of $\overline{\Q}_\ell\to\C$ and by the purity of weight $0$ of the sheaves $\mathcal{F}_{1,p},\mathcal{F}_{2,p},$ we have that $\chi$ is either a trivial or a quadratic character. 
    
    Now, take $d'$ to be a nonsquare in $\F_p$ and consider $d=(d')^{\ord(\chi)}.$ By comparing the traces of Frobenius of $\mathcal{F}_{2,p}\otimes\mathcal{L}_p$ and $\mathcal{F}_{2,d,p},$  the quasi-orthogonality relations of irreducible sheaves (see Theorem 5.2 of \cite{AppliedladicCohomology}) imply that, up to enlarging $N,$ we have $\mathcal{F}_{2,p}\otimes\mathcal{L}_p\cong\mathcal{F}_{2,d,p}$ for all primes $p\nmid N.$
\end{proof}

\begin{proof}[Proof of Proposition~\ref{IsogenyCondition}]
    This immediately follows from Lemmas~\ref{TateSheaf}, \ref{geomConstant}, \ref{geomConstantQuad}, and Theorem~\ref{Parsin}.
\end{proof}

\section{Some Distributions}\label{Section7}

To obtain Corollaries~\ref{corollary1} and \ref{corollary2}, we combine Theorems~\ref{theorem1} and \ref{theorem2} with the method of moments. In this section, we compute the required probability density functions associated to a given sequence of moments or mixed moments.

For the mixed moments, the following lemma gives the probability density functions.

\begin{lemma}\label{mixedMomentsPDF}
    If $p$ is prime and $i=1,2,$ consider maps
    $$
    f_{i,p}:\F_p\to [-2,2].
    $$
    Suppose that for each $m,n\geq 0$ we have that
    $$
    \lim\limits_{p\to\infty}\frac{1}{p}\sum\limits_{\lambda\in\F_p}f_{1,p}(\lambda)^mf_{2,p}(\lambda)^n =\begin{cases}
        C(m_1)C(n_1) &\ \ \ \text{ if } m=2m_1\text{ and }n=2n_1\text{ are even}\\
        0 &\ \ \ \text{ otherwise}.
    \end{cases}
    $$
    If $-2\leq a_i,b_i\leq 2,$ then we have
    $$
    \lim\limits_{p\to\infty}\frac{\#\{\lambda\in\F_p: (f_{1,p}(\lambda),f_{2,p}(\lambda))\in[a_1,b_1]\times[a_2,b_2]\}}{p}=\frac{1}{4\pi^2}\int_{a_1}^{b_1}\int_{a_2}^{b_2}\sqrt{4-x_1^2}\cdot\sqrt{4-x_2^2} \, dx_1dx_2.
    $$
\end{lemma}

\begin{proof}[Sketch of Proof]
    
    The method of moments (see Exercises 30.5 and 30.6 of \cite{Billingsley}) states that in order to compute the limiting distribution, under very general settings, it suffices to show that the limiting moments coincide with the moments of the claimed distribution. In our case, this is precisely the statement that if $m,n\geq 0,$ then we have that 
    $$
    \int_{-2}^2\int_{-2}^2x_1^mx_2^n\sqrt{4-x_1^2}\sqrt{4-x_2^2}dx_1dx_2=\begin{cases}
        C(m_1)C(n_1) &\ \ \ \text{ if } m=2m_1\text{ and }n=2n_1\text{ are even}\\
        0 &\ \ \ \text{ otherwise,}
    \end{cases}
    $$
    which is straightforward.
\end{proof}

We now determine the probability density function for the moments arising in Theorem~\ref{theorem1}. To this end, we recall the definition of the Meijer $G$-function (see Appendix of \cite{IntTransforms}).

\begin{definition}\label{MeijerGDef}
    Let $0\leq m\leq q,0\leq n\leq p$ be integers and suppose $a_1,\ldots,a_p,b_1,\ldots,b_q$ are complex numbers such that $a_k-b_j\not\in\Z_{>0}$ for all $1\leq k\leq n$ and $1\leq j\leq m.$ Furthermore, suppose $\Re(\nu)<-1,$ where
    $$
    \nu:=\sum\limits_{j=1}^q b_j-\sum\limits_{j=1}^pa_p.
    $$
    In this notation if $z\in\C^\times$ and $|z|\leq 1,$ the Meijer $G$-function is given as
    $$
    G_{p,q}^{m,n} \left[\begin{matrix}a_1 & \ldots & a_p \\  b_1 & \ldots & b_q\end{matrix} \; \bigg| \; z \right]:=\frac{1}{2\pi i}\bigintsss_L \frac{\prod\limits_{j=1}^m\Gamma(b_j-s)\prod\limits_{j=1}^n\Gamma(1-a_j+s)}{\prod\limits_{j=m+1}^q\Gamma(1-b_j+s)\prod\limits_{j=n+1}^p\Gamma(a_j-s)}z^sds,
    $$
    where $\Gamma(\cdot)$ denotes Euler's Gamma function, and where $L$ is a loop beginning and ending at $-\infty$ that encircles all poles of $\Gamma(1-a_k+s)$ for $1\leq k\leq n$ exactly once in the positive direction and does not encircle any poles of $\Gamma(b_j-s)$ for $1\leq j\leq m.$
\end{definition}

In order to compute integrals involving Meijer $G$-functions, we require the following key transformation and integral identities.

\begin{lemma}\label{MeijerGProps}
    The following are true.
    \begin{enumerate}
    \item If $\rho\in \C,$ then we have
    $$
    z^{\rho} \cdot \, G_{p,q}^{m,n} \left[\begin{matrix} a_1 & \ldots & a_p \\  b_{1} & \ldots & b_q\end{matrix} \; \bigg| \; z \right] = G_{p,q}^{m,n}\left[\begin{matrix} a_1+\rho & \ldots & a_p+\rho \\  b_1+\rho & \ldots & b_q+\rho\end{matrix} \; \bigg| \; z \right].
    $$
    \item Under general conditions (see 20.5 of \cite{IntTransforms}), if $\alpha,\beta,z\in\R,$ then we have
    $$
    \int_{0}^{1} x^{-\alpha}(1-x)^{\alpha-\beta-1} G_{p,q}^{m,n} \left[\begin{matrix} a_1 & \ldots & a_p \\  b_1 & \ldots & b_q\end{matrix} \; \bigg| \; zx \right] dx = \Gamma(\alpha-\beta)\cdot G_{p+1,q+1}^{m,n+1} \left[\begin{matrix} \alpha & a_{1} & \ldots & a_p \\  b_{1} & \ldots & b_q & \beta\end{matrix} \; \bigg| \; z \right].
    $$

    \end{enumerate}
\end{lemma}

We will also make use of the following elementary lemma which expresses products of Catalan numbers in terms of ratios of Gamma functions.

\begin{lemma}\label{CatalanGamma}
    If $m \geq 0,$ then we have
    $$
     C(m)C(m+1) = \frac{4 \cdot 16^{m}}{\pi} \frac{\Gamma\left(m+\frac{1}{2}\right)\Gamma\left(m+\frac{3}{2}\right)}{\Gamma(m+2)\Gamma(m+3)}.
    $$
\end{lemma}

\begin{proof}

Using the property that $\Gamma(x+1)=x\Gamma(x)$ and that $\Gamma(\frac{1}{2})=\sqrt{\pi},$ we have that
$$
\Gamma\left(m+\frac{1}{2}\right) = \frac{(2m-1)!!}{2^m} \sqrt{\pi},
$$
where $(2k-1)!!:=(2k-1)(2k-3)\cdot\ldots\cdot 1.$ Therefore, we have that

\begin{equation*}
\begin{split}
\frac{4^{m}\Gamma\left(m+\frac{1}{2}\right)}{\Gamma(m+2)\sqrt{\pi}} &= \frac{2^{m}(2m-1)!!}{(m+1)!}\\
&= \frac{1}{m+1} \binom{2m}{m}\\
&= C(m)
\end{split}
\end{equation*}
and our claim follows immediately.
\end{proof}

Finally, we compute the probability density function associated to the moments that appear in Theorem~\ref{theorem1}.

\begin{lemma}\label{momentsPDF}
    If $p$ is prime, consider maps
    $$
    f_{p}:\F_p\to [-4,4].
    $$
    Suppose that for each $m\geq 0$ we have that
    $$
    \lim\limits_{p\to\infty}\frac{1}{p}\sum\limits_{\lambda\in\F_p}f_{p}(\lambda)^m=\begin{cases}
        C(m_1)C(m_1+1) &\ \ \ \text{ if } m=2m_1\text{ is even}\\
        0 &\ \ \ \text{ otherwise}.
    \end{cases}
    $$
    If $-4\leq a< b\leq 4,$ then we have
    $$
    \lim\limits_{p\to\infty}\frac{\#\left\{\lambda\in\F_p : f_p(\lambda)\in[a,b] \right\}}{p} = \int_a^b\frac{4}{\pi |t|} G_{2,2}^{2,0} \left[\begin{matrix} 2 & 3 \smallskip \\  \frac{1}{2} & \frac{3}{2}\end{matrix} \; \bigg| \; \frac{t^2}{16} \right]dt.
    $$
\end{lemma}

\begin{proof}
    By the method of moments (see Theorems 30.5 and 30.6 of \cite{Billingsley}), it suffices to show that for $m\geq 0,$ we have
    $$
    \int_{-4}^4 t^m\cdot\frac{4}{\pi |t|} G_{2,2}^{2,0} \left[\begin{matrix} 2 & 3 \smallskip \\  \frac{1}{2} &  \frac{3}{2}\end{matrix} \; \bigg| \; \frac{t^2}{16}\right]dt= \begin{cases}
        C(m_1)C(m_1+1) &\ \ \ \text{ if } m=2m_1\text{ is even}\\
        0 &\ \ \ \text{ otherwise}.
    \end{cases}.
    $$
    If $m$ is odd, this is clear since the integrand is odd. If $m=2m_1,$ the integrand is even and we rewrite the integral as
    $$
    \int_{-4}^4 t^m\cdot\frac{4}{\pi |t|} G_{2,2}^{2,0} \left[\begin{matrix} 2 & 3 \smallskip \\  \frac{1}{2} & \frac{3}{2}\end{matrix} \; \bigg| \; \frac{t^2}{16}\right]dt = \frac{2\cdot 16^{m_1}}{\pi}\int_{0}^{4} \left(\frac{t^2}{16}\right)^{m_1-\frac{1}{2}} \cdot G_{2,2}^{2,0}\left[\begin{matrix}2 & 3\smallskip \\  \frac{1}{2} & \frac{3}{2}\end{matrix} \; \bigg| \; \frac{t^2}{16} \right] \, dt
    $$
    Using (1) of Lemma~\ref{MeijerGProps} and the change-of-variables $\omega=\frac{t^2}{16},$ we have
    $$
     \int_{-4}^4 t^m\cdot\frac{4}{\pi |t|} G_{2,2}^{2,0} \left[\begin{matrix} 2 & 3 \smallskip \\  \frac{1}{2} & \frac{3}{2}\end{matrix} \; \bigg| \; \frac{t^2}{16}\right]dt =\frac{4 \cdot 16^{m_1}}{\pi} \int_{0}^{1} \omega^{-\frac{1}{2}} \cdot G_{2,2}^{2,0}\left[\begin{matrix}m_1 + \frac{3}{2} &  m_1 + \frac{5}{2}\smallskip \\  m_1 & m_1 + 1\end{matrix} \; \bigg| \; \omega \right] d \omega
    $$
    Computing the integral using (2) of Lemma~\ref{MeijerGProps}, we obtain
    $$
    \int_{-4}^4 t^m\cdot\frac{4}{\pi |t|} G_{2,2}^{2,0} \left[\begin{matrix} 2 & 3 \smallskip \\  \frac{1}{2} & \frac{3}{2}\end{matrix} \; \bigg| \; \frac{t^2}{16}\right]dt = \frac{4 \cdot 16^{m_1}}{\pi} G_{3,3}^{2,1} \left[\begin{matrix}\frac{1}{2} & m_1 + \frac{3}{2} & m_1 + \frac{5}{2}\smallskip \\  m_1 & m_1 + 1 & - \frac{1}{2}\end{matrix} \; \bigg| \; 1 \right]
    $$
    Applying (1) of Lemma~\ref{MeijerGProps} with $z=1$ and $\rho=-m_1,$ we obtain
    $$
    \int_{-4}^4 t^m\cdot\frac{4}{\pi |t|} G_{2,2}^{2,0} \left[\begin{matrix} 2 & 3 \smallskip \\  \frac{1}{2} & \frac{3}{2}\end{matrix} \; \bigg| \; \frac{t^2}{16}\right]dt = \frac{4 \cdot 16^{m_1}}{\pi} G_{3,3}^{2,1} \left[\begin{matrix}\frac{1}{2} - m_1 & \frac{3}{2} & \frac{5}{2}\smallskip \\  0 & 1 & -m_1 - \frac{1}{2}\end{matrix} \; \bigg| \; 1 \right]
    $$
    The residue theorem then implies that
\begin{equation*}
\begin{split}
G_{3,3}^{2,1} \left[\begin{matrix}\frac{1}{2} - m_1 & \frac{3}{2} & \frac{5}{2}\smallskip \\  0 & 1 & -m_1 - \frac{1}{2}\end{matrix} \; \bigg| \; 1 \right] &= \frac{1}{2 \pi i} \int_{L} \frac{\Gamma\left(\frac{1}{2}+m_1+s\right)}{\Gamma\left(\frac{3}{2}+m_1+s\right)} \cdot \frac{\Gamma(-s)\Gamma(1-s)}{\Gamma\left(\frac{3}{2}-s\right)\Gamma\left(\frac{5}{2}-s\right)} ds\\
&= \frac{1}{2 \pi i} \int_{L} \frac{1}{\frac{1}{2}+m_1+s} \cdot \frac{\Gamma(-s)\Gamma(1-s)}{\Gamma\left(\frac{3}{2}-s\right)\Gamma\left(\frac{5}{2}-s\right)} ds\\
&= \underset{s = -\frac{1}{2}-m_1}{\Res} \left(\frac{1}{\frac{1}{2}+m_1+s} \cdot \frac{\Gamma(-s)\Gamma(1-s)}{\Gamma\left(\frac{3}{2}-s\right)\Gamma\left(\frac{5}{2}-s\right)} \right)\\
&= \frac{\Gamma\left(\frac{1}{2}+m_1\right)\Gamma\left(\frac{3}{2}+m_1\right)}{\Gamma(2+m_1)\Gamma(3+m_1)}.
\end{split}
\end{equation*}
By Lemma~\ref{CatalanGamma}, our claim follows.
\end{proof}

\section{Proofs of the Main Results}\label{Section8}

\begin{proof}[Proof of Theorem~\ref{theorem1}]
    $E_\lambda^{\Leg}$ has multiplicative reduction at $\lambda=1$ whereas $E_{-\lambda}^{\Leg}$ has good reduction at $1.$ Therefore, Proposition~\ref{mixedMoments} applied to $(E_\lambda^{\Leg},E_{-\lambda}^{\Leg})$ combined with Lemma~\ref{CombMom} gives the claimed moments.
\end{proof}

\begin{proof}[Proof of Corollary~\ref{corollary1}]
    This follows from Theorem~\ref{theorem1} and Lemma~\ref{momentsPDF}.
\end{proof}

\begin{proof}[Proof of Theorem~\ref{theorem2}]
    Proposition~\ref{IsogenyCondition} shows that we can use Proposition~\ref{mixedMoments} whenever we have a generic pair, and the latter gives the moments.
\end{proof}

\begin{proof}[Proof of Corollary~\ref{corollary2}]
    This follows from Theorem~\ref{theorem2} and Lemma~\ref{mixedMomentsPDF}.
\end{proof}

\begin{proof}[Proof of Theorem~\ref{theorem3}]
This follows from Lemma~\ref{generalizedKoike} and Proposition~\ref{mixedMoments} with $n=0$ (or Section 3 of \cite{RangMoyen}). 
\end{proof}

\begin{proof}[Proof of Theorem~\ref{theorem4}]
    This follows from Lemma~\ref{hypergeometricClausen} and Proposition~\ref{3F2MomEC}.
\end{proof}

\bibliography{DraftReferences}

\newcommand{\etalchar}[1]{$^{#1}$}
\providecommand{\bysame}{\leavevmode\hbox to3em{\hrulefill}\thinspace}
\providecommand{\MR}{\relax\ifhmode\unskip\space\fi MR }
\providecommand{\MRhref}[2]{%
  \href{http://www.ams.org/mathscinet-getitem?mr=#1}{#2}
}
\providecommand{\href}[2]{#2}
\begin{thebibliography}{EMOT54}

\bibitem[AGLT11]{HMM1}
Michael Allen, Brian Grove, Ling Long, and Fang-Ting Tu, \emph{The {E}xplicit {H}ypergeometric-{M}odularity {M}ethod {I}}, 2024 arXiv: 2404.00711.

\bibitem[AGLT16]{HMM2}
Michael Allen, Brian Grove, Ling Long, and Fang-Ting Tu, \emph{The {E}xplicit {H}ypergeometric-{M}odularity {M}ethod {II}}, 2024 arXiv: 2411.15116.

\bibitem[Ahl01]{CS1}
Scott Ahlgren, \emph{Gaussian hypergeometric series and combinatorial congruences}, Symbolic computation, number theory, special functions, physics and combinatorics ({G}ainesville, {FL}, 1999), Dev. Math., vol.~4, Kluwer Acad. Publ., Dordrecht, 2001, pp.~1--12. \MR{1880076}

\bibitem[AO00a]{CS2}
Scott Ahlgren and Ken Ono, \emph{A {G}aussian hypergeometric series evaluation and {A}p\'ery number congruences}, J. Reine Angew. Math. \textbf{518} (2000), 187--212. \MR{1739404}

\bibitem[AO00b]{CY2}
\bysame, \emph{Modularity of a certain {C}alabi-{Y}au threefold}, Monatsh. Math. \textbf{129} (2000), no.~3, 177--190. \MR{1746757}

\bibitem[AOP02]{K31}
Scott Ahlgren, Ken Ono, and David Penniston, \emph{Zeta functions of an infinite family of {$K3$} surfaces}, Amer. J. Math. \textbf{124} (2002), no.~2, 353--368. \MR{1890996}

\bibitem[BCM15]{BCM}
Frits Beukers, Henri Cohen, and Anton Mellit, \emph{Finite hypergeometric functions}, Pure Appl. Math. Q. \textbf{11} (2015), no.~4, 559--589. \MR{3613122}

\bibitem[Bil12]{Billingsley}
Patrick Billingsley, \emph{Probability and measure}, anniversary ed., Wiley Series in Probability and Statistics, John Wiley \& Sons, Inc., Hoboken, NJ, 2012, With a foreword by Steve Lalley and a brief biography of Billingsley by Steve Koppes. \MR{2893652}

\bibitem[BK13]{EC1}
Rupam Barman and Gautam Kalita, \emph{Hypergeometric functions over {$\mathbb{F}_q$} and traces of {F}robenius for elliptic curves}, Proc. Amer. Math. Soc. \textbf{141} (2013), no.~10, 3403--3410. \MR{3080163}

\bibitem[BKS14]{hyper1}
Rupam Barman, Gautam Kalita, and Neelam Saikia, \emph{Hyperelliptic curves and values of {G}aussian hypergeometric series}, Arch. Math. (Basel) \textbf{102} (2014), no.~4, 345--355. \MR{3196962}

\bibitem[BRS16]{Dwork1}
Rupam Barman, Hasanur Rahman, and Neelam Saikia, \emph{Counting points on {D}work hypersurfaces and {$p$}-adic hypergeometric functions}, Bull. Aust. Math. Soc. \textbf{94} (2016), no.~2, 208--216. \MR{3568913}

\bibitem[BSM15]{hyper2}
Rupam Barman, Neelam Saikia, and Dermot McCarthy, \emph{Summation identities and special values of hypergeometric series in the {$p$}-adic setting}, J. Number Theory \textbf{153} (2015), 63--84. \MR{3327565}

\bibitem[Del80]{WeilII}
Pierre Deligne, \emph{La conjecture de {W}eil. {II}}, Inst. Hautes \'Etudes Sci. Publ. Math. (1980), no.~52, 137--252. \MR{601520}

\bibitem[DKS{\etalchar{+}}20]{K32}
Charles~F. Doran, Tyler~L. Kelly, Adriana Salerno, Steven Sperber, John Voight, and Ursula Whitcher, \emph{Hypergeometric decomposition of symmetric {K}3 quartic pencils}, Res. Math. Sci. \textbf{7} (2020), no.~2, Paper No. 7, 81. \MR{4078177}

\bibitem[EMOT54]{IntTransforms}
A.~Erd\'elyi, W.~Magnus, F.~Oberhettinger, and F.~G. Tricomi, \emph{Tables of integral transforms. {V}ol. {II}}, McGraw-Hill Book Co., Inc., New York-Toronto-London, 1954, Based, in part, on notes left by Harry Bateman. \MR{65685}

\bibitem[FKM14]{traceFunctionsFFApplications}
\'Etienne Fouvry, Emmanuel Kowalski, and Philippe Michel, \emph{Trace functions over finite fields and their applications}, Colloquium {D}e {G}iorgi 2013 and 2014, Colloquia, vol.~5, Ed. Norm., Pisa, 2014, pp.~7--35. \MR{3379177}

\bibitem[FKM15]{StudySumsProducts}
\bysame, \emph{A study in sums of products}, Philos. Trans. Roy. Soc. A \textbf{373} (2015), no.~2040, 20140309, 26. \MR{3338119}

\bibitem[FKMS19]{AppliedladicCohomology}
\'Etienne Fouvry, Emmanuel Kowalski, Philippe Michel, and Will Sawin, \emph{Lectures on applied {$\ell$}-adic cohomology}, Analytic methods in arithmetic geometry, Contemp. Math., vol. 740, Amer. Math. Soc., [Providence], RI, [2019] \copyright 2019, pp.~113--195. \MR{4033731}

\bibitem[FKRS12]{SatoTateGroups}
Francesc Fit\'e, Kiran~S. Kedlaya, V\'ictor Rotger, and Andrew~V. Sutherland, \emph{Sato-{T}ate distributions and {G}alois endomorphism modules in genus 2}, Compos. Math. \textbf{148} (2012), no.~5, 1390--1442. \MR{2982436}

\bibitem[FLR{\etalchar{+}}22]{FusEtAl}
Jenny Fuselier, Ling Long, Ravi Ramakrishna, Holly Swisher, and Fang-Ting Tu, \emph{Hypergeometric functions over finite fields}, Mem. Amer. Math. Soc. \textbf{280} (2022), no.~1382, vii+124. \MR{4493579}

\bibitem[FM16]{MM3}
Jenny~G. Fuselier and Dermot McCarthy, \emph{Hypergeometric type identities in the {$p$}-adic setting and modular forms}, Proc. Amer. Math. Soc. \textbf{144} (2016), no.~4, 1493--1508. \MR{3451227}

\bibitem[FOP04]{MM1}
Sharon Frechette, Ken Ono, and Matthew Papanikolas, \emph{Gaussian hypergeometric functions and traces of {H}ecke operators}, Int. Math. Res. Not. (2004), no.~60, 3233--3262. \MR{2096220}

\bibitem[Fus10]{EC2}
Jenny~G. Fuselier, \emph{Hypergeometric functions over {$\mathbb{F}_p$} and relations to elliptic curves and modular forms}, Proc. Amer. Math. Soc. \textbf{138} (2010), no.~1, 109--123. \MR{2550175}

\bibitem[Fus13]{MM2}
\bysame, \emph{Traces of {H}ecke operators in level 1 and {G}aussian hypergeometric functions}, Proc. Amer. Math. Soc. \textbf{141} (2013), no.~6, 1871--1881. \MR{3034414}

\bibitem[Goo17]{Dwork3}
Heidi Goodson, \emph{Hypergeometric functions and relations to {D}work hypersurfaces}, Int. J. Number Theory \textbf{13} (2017), no.~2, 439--485. \MR{3606631}

\bibitem[Gre84]{GreenePhD}
John~Robert Greene, \emph{C{haracter} {sum} {analogues} {for} {hypergeometric} {and} {generalized} {hypergeometric} {functions} {over} {finite} {fields}}, ProQuest LLC, Ann Arbor, MI, 1984, Thesis (Ph.D.)--University of Minnesota. \MR{2633473}

\bibitem[Gre87]{GreenePaper}
John Greene, \emph{Hypergeometric functions over finite fields}, Trans. Amer. Math. Soc. \textbf{301} (1987), no.~1, 77--101. \MR{879564}

\bibitem[Gro24]{grove2024hypergeometricmomentshecketrace}
Brian Grove, \emph{Hypergeometric {M}oments and {H}ecke {T}race {F}ormulas}, 2024.

\bibitem[HLLT18]{HLLT}
Jerome~W. Hoffman, Wen-Ching~Winnie Li, Ling Long, and Fang-Ting Tu, \emph{Traces of {H}ecke {O}perators via {H}ypergeometric {C}haracter {S}ums}, 2024z; arXiv: 2408.02918.

\bibitem[HT20]{HM2}
J.~William Hoffman and Fang-Ting Tu, \emph{Transformations of hypergeometric motives}, 2020.

\bibitem[Jac13]{EnumCombinatorics}
David Jackson, \emph{{\it {E}numerative combinatorics. {V}olume {I}.} {S}econd edition [book review of mr2868112]}, SIAM Rev. \textbf{55} (2013), no.~1, 193--194. \MR{3077474}

\bibitem[Kat88]{KatzGKM}
Nicholas~M. Katz, \emph{Gauss sums, {K}loosterman sums, and monodromy groups}, Annals of Mathematics Studies, vol. 116, Princeton University Press, Princeton, NJ, 1988. \MR{955052}

\bibitem[Kat90a]{KatzESDE}
\bysame, \emph{Exponential sums and differential equations}, Annals of Mathematics Studies, vol. 124, Princeton University Press, Princeton, NJ, 1990. \MR{1081536}

\bibitem[Kat90b]{KatzESFDEC}
\bysame, \emph{Exponential sums over finite fields and differential equations over the complex numbers: some interactions}, Bull. Amer. Math. Soc. (N.S.) \textbf{23} (1990), no.~2, 269--309. \MR{1032857}

\bibitem[KC17]{CS3}
Gautam Kalita and Arjun~Singh Chetry, \emph{Congruences for generalized {A}p\'ery numbers and {G}aussian hypergeometric series}, Res. Number Theory \textbf{3} (2017), Paper No. 5, 15. \MR{3615167}

\bibitem[Koi95]{Koike}
Masao Koike, \emph{Orthogonal matrices obtained from hypergeometric series over finite fields and elliptic curves over finite fields}, Hiroshima Math. J. \textbf{25} (1995), no.~1, 43--52. \MR{1322601}

\bibitem[Kow14]{KowalskiRep}
Emmanuel Kowalski, \emph{An introduction to the representation theory of groups}, Graduate Studies in Mathematics, vol. 155, American Mathematical Society, Providence, RI, 2014. \MR{3236265}

\bibitem[KS99]{KatzSarnak}
Nicholas~M. Katz and Peter Sarnak, \emph{Random matrices, {F}robenius eigenvalues, and monodromy}, American Mathematical Society Colloquium Publications, vol.~45, American Mathematical Society, Providence, RI, 1999. \MR{1659828}

\bibitem[Len11a]{EC3}
Catherine Lennon, \emph{Gaussian hypergeometric evaluations of traces of {F}robenius for elliptic curves}, Proc. Amer. Math. Soc. \textbf{139} (2011), no.~6, 1931--1938. \MR{2775369}

\bibitem[Len11b]{MM4}
\bysame, \emph{Trace formulas for {H}ecke operators, {G}aussian hypergeometric functions, and the modularity of a threefold}, J. Number Theory \textbf{131} (2011), no.~12, 2320--2351. \MR{2832827}

\bibitem[LTYZ21]{CY1}
Ling Long, Fang-Ting Tu, Noriko Yui, and Wadim Zudilin, \emph{Supercongruences for rigid hypergeometric {C}alabi-{Y}au threefolds}, Adv. Math. \textbf{393} (2021), Paper No. 108058, 49. \MR{4330088}

\bibitem[LZ10]{TemperleyLieb}
G.~I. Lehrer and R.~B. Zhang, \emph{A {T}emperley-{L}ieb analogue for the {BMW} algebra}, Representation theory of algebraic groups and quantum groups, Progr. Math., vol. 284, Birkh\"auser/Springer, New York, 2010, pp.~155--190. \MR{2761939}

\bibitem[McC12a]{CS4}
Dermot McCarthy, \emph{On a supercongruence conjecture of {R}odriguez-{V}illegas}, Proc. Amer. Math. Soc. \textbf{140} (2012), no.~7, 2241--2254. \MR{2898688}

\bibitem[McC12b]{McCarthyTransformations}
\bysame, \emph{Transformations of well-poised hypergeometric functions over finite fields}, Finite Fields Appl. \textbf{18} (2012), no.~6, 1133--1147. \MR{3019189}

\bibitem[McC13]{EC4}
\bysame, \emph{The trace of {F}robenius of elliptic curves and the {$p$}-adic gamma function}, Pacific J. Math. \textbf{261} (2013), no.~1, 219--236. \MR{3037565}

\bibitem[McC17]{Dwork2}
\bysame, \emph{The number of {$\mathbb{F}_p$}-points on {D}work hypersurfaces and hypergeometric functions}, Res. Math. Sci. \textbf{4} (2017), Paper No. 4, 15. \MR{3630724}

\bibitem[Mic95]{RangMoyen}
Philippe Michel, \emph{Rang moyen de familles de courbes elliptiques et lois de {S}ato-{T}ate}, Monatsh. Math. \textbf{120} (1995), no.~2, 127--136. \MR{1348365}

\bibitem[Mil80]{ECMilneBook}
James~S. Milne, \emph{\'etale cohomology}, Princeton Mathematical Series, vol. No. 33, Princeton University Press, Princeton, NJ, 1980. \MR{559531}

\bibitem[Mil86]{MilneAVCollection}
J.~S. Milne, \emph{Abelian varieties}, Arithmetic geometry ({S}torrs, {C}onn., 1984), Springer, New York, 1986, pp.~103--150. \MR{861974}

\bibitem[MOS20]{CS5}
Dermot McCarthy, Robert Osburn, and Armin Straub, \emph{Sequences, modular forms and cellular integrals}, Math. Proc. Cambridge Philos. Soc. \textbf{168} (2020), no.~2, 379--404. \MR{4064111}

\bibitem[MP15]{MM5}
Dermot McCarthy and Matthew~A. Papanikolas, \emph{A finite field hypergeometric function associated to eigenvalues of a {S}iegel eigenform}, Int. J. Number Theory \textbf{11} (2015), no.~8, 2431--2450. \MR{3420754}

\bibitem[MT24]{splittingRoots}
Dermot McCarthy and Mohit Tripathi, \emph{Splitting hypergeometric functions over roots of unity}, Res. Math. Sci. \textbf{11} (2024), no.~3, Paper No. 57, 26. \MR{4789405}

\bibitem[Ono98]{EC5}
Ken Ono, \emph{Values of {G}aussian hypergeometric series}, Trans. Amer. Math. Soc. \textbf{350} (1998), no.~3, 1205--1223. \MR{1407498}

\bibitem[OPSS24]{ono2024distributionhessianvaluesgaussian}
Ken Ono, Sudhir Pujahari, Hasan Saad, and Neelam Saikia, \emph{Distribution of the hessian values of gaussian hypergeometric functions}, 2024.

\bibitem[OSS23]{HGFDistributionOnoSaikia}
Ken Ono, Hasan Saad, and Neelam Saikia, \emph{Distribution of values of {G}aussian hypergeometric functions}, Pure Appl. Math. Q. \textbf{19} (2023), no.~1, 371--407. \MR{4570167}

\bibitem[OSZ18]{CS6}
Robert Osburn, Armin Straub, and Wadim Zudilin, \emph{A modular supercongruence for {${}_6F_5$}: an {A}p\'ery-like story}, Ann. Inst. Fourier (Grenoble) \textbf{68} (2018), no.~5, 1987--2004. \MR{3893762}

\bibitem[Pap06]{MM6}
Matthew Papanikolas, \emph{A formula and a congruence for {R}amanujan's {$\tau$}-function}, Proc. Amer. Math. Soc. \textbf{134} (2006), no.~2, 333--341. \MR{2175999}

\bibitem[Par72]{Parsin}
A.~N. Par{\v{s}}in, \emph{Minimal models of curves of genus {$2$}, and homomorphisms of abelian varieties defined over a field of finite characteristic}, Izv. Akad. Nauk SSSR Ser. Mat. \textbf{36} (1972), 67--109. \MR{316456}

\bibitem[PV04]{pastur2004moments}
Leonid Pastur and Vladimir Vasilchuk, \emph{On the moments of traces of matrices of classical groups}, Communications in mathematical physics \textbf{252} (2004), 149--166.

\bibitem[RRV22]{HM1}
David~P. Roberts and Fernando Rodriguez~Villegas, \emph{Hypergeometric motives}, Notices Amer. Math. Soc. \textbf{69} (2022), no.~6, 914--929. \MR{4442789}

\bibitem[Saa23]{3F2Explicit}
Hasan Saad, \emph{Explicit {S}ato-{T}ate type distribution for a family of {$K3$} surfaces}, Forum Math. \textbf{35} (2023), no.~4, 1105--1132. \MR{4609844}

\bibitem[Sil09]{SilvermanAEC}
Joseph~H. Silverman, \emph{The arithmetic of elliptic curves}, second ed., Graduate Texts in Mathematics, vol. 106, Springer, Dordrecht, 2009. \MR{2514094}

\bibitem[Sta15]{CatalanBook}
Richard~P. Stanley, \emph{Catalan numbers}, Cambridge University Press, New York, 2015. \MR{3467982}

\bibitem[TY18]{EC6}
Fang-Ting Tu and Yifan Yang, \emph{Evaluation of certain hypergeometric functions over finite fields}, SIGMA Symmetry Integrability Geom. Methods Appl. \textbf{14} (2018), Paper No. 050, 18. \MR{3803730}

\end{thebibliography}
\bibliographystyle{amsalpha}

\end{document}